\documentclass{article}
\usepackage[utf8]{inputenc}
\usepackage{amsthm}
\usepackage[a4paper,lmargin=3cm,rmargin=2cm,tmargin=3cm,bmargin=2cm]{geometry}
\usepackage{amssymb} 
\usepackage{kpfonts}
\usepackage{centernot}
\usepackage[dvipsnames]{xcolor}
\usepackage{amsmath}
\usepackage{enumitem}
\usepackage{bbm}
\usepackage{indentfirst}
\usepackage{comment}
\usepackage{tikz}
\usetikzlibrary{patterns,hobby}
\usepackage{lipsum,lmodern}
\usepackage[most]{tcolorbox}
\usepackage{placeins}
\usepackage{graphicx}
\usepackage{wasysym}
\usepackage[export]{adjustbox}
\usepackage{titling}
\usepackage{ulem}
\usepackage[colorlinks=true, allcolors=blue]{hyperref}
\newcommand{\constMarcosBeta}{\beta}
\newcommand{\constMarcosAlfa}{\alpha}
\newcommand{\constMarcosGamma}{\gamma}
\newcommand{\constMarcosMu}{\mu}

\newcommand{\constDecaimento}{\phi}
\newcommand{\constNossoBeta}{\rho}
\newcommand{\constNossoMu}{\psi}

\DeclareRobustCommand{\looongleftrightarrow}{
\leftarrow\joinrel\DOTSB\relbar\joinrel\relbar\joinrel\relbar\joinrel\rightarrow}

\renewcommand{\P}{\mathbb{P}}
\renewcommand{\S}{\mathbb{S}}
\newcommand{\Z}{\mathbb{Z}}
\newcommand{\E}{\mathcal{E}}
\newcommand{\Q}{\mathbb{Q}}
\newcommand{\slab}{\mathbf}
\newcommand{\first}{\textbf}
\newcommand{\frontA}{\partial}
\newcommand{\frontB}{\partial}
\newcommand{\frontC}{\partial^{c}}
\newcommand{\frontD}{\partial^{out}}

\DeclareRobustCommand{\looooongleftrightarrow}{\leftarrow\joinrel\relbar\joinrel\relbar\joinrel\relbar\joinrel\relbar\joinrel\relbar\joinrel\relbar\joinrel\relbar\joinrel\rightarrow}
\newtheorem{theorem}{Theorem}
\newtheorem{lema}{Lemma}
\newtheorem{proposition}{Proposition}

\newtheorem{definition}{Definition}

\newtheorem{corollary}{Corollary}[theorem]
\newtheorem{claim}{Claim}

\newtheorem{remark}{Remark}
\tikzstyle{modSK} = [rectangle, rounded corners, minimum width=3cm, minimum height=1cm,text centered, draw=black, fill=blue!30, text width=3.5cm,font=\footnotesize]
\tikzstyle{modZ2Dep} = [rectangle, rounded corners, minimum width=3cm, minimum height=1cm,text centered, draw=black, fill=orange!30, text width=4.5cm,font=\footnotesize]
\tikzstyle{modZ2inDep} = [rectangle, rounded corners, minimum width=3cm, minimum height=1cm,text centered, draw=black, fill=green!30, text width=3.5cm,font=\footnotesize]
\tikzstyle{arrow} = [thick,->,>=stealth]

\title{Critical percolation on slabs with random columnar disorder}

\begin{document}
\author{
    Matheus B. Castro\thanks{Departamento de Matemática, Universidade Federal de Minas Gerais, Belo Horizonte, Brazil. \newline}\,\,\,\,\,\,\,\,\,\,\,\,\,
    Rémy Sanchis$^*$\,\,\,\,\,\,\,\,\,\,\,\,\,
    Roger W. C. Silva\thanks{Departamento de Estatística, Universidade Federal de Minas Gerais, Belo Horizonte, Brazil.\newline}
}
\date{}
\maketitle
{\centering\small \textit{This paper is dedicated to the memory of Elisvaldo M. Silva}\par}
\vspace{1cm}
\begin{abstract}
We explore critical Bernoulli percolation on three-dimensional slabs $\S^+_k=\mathbb{Z}_+^2\times\{0,\dots,k\}$ featuring one-dimensional random reinforcements.  A vertical column indexed by $x\in\Z_+$ is the set $\{x\}\times\mathbb{Z}_+\times\{0\}$. Vertical columns are selected based on the arrivals of a renewal process given by i.i.d. copies of a random variable $\xi$ with $\mathbb{E}(\xi^{\phi})<\infty$, for some $\phi>1$. Edges on selected vertical columns are independently open with probability $q$ and closed with probability  $1-q$. All remaining edges are independently open with probability $p$ and closed with probability  $1-p$. We prove that for all sufficiently large $\constDecaimento$  (depending solely on $k$), the following assertion holds: for every $q>p_c(\S^+_k)$, one can take $p$ strictly smaller than $p_c(\S^+_k)$ such that percolation still occurs. We also derive a new upper bound on the correlation length for Bernoulli percolation on $\S_k^+$. \\


\end{abstract}

\section{Introduction}

\subsection{Background and motivation}

Percolation theory is a fundamental field in statistical physics and mathematics, dealing with the behavior of connected clusters in a random graph. Since its introduction by Broadbent and Hammersley in 1957 \cite{BH}, the theory has inspired a variety of models that capture the essence of connectivity and randomness in diverse systems.

In classical percolation on the usual hypercubic lattice $\Z^d$, each edge $e$ is assigned a state $\omega(e)$, either 0 or 1, where the $\omega(e)'s$ are independent Bernoulli random variables with mean $p$. The configuration $\omega$ defines a random subgraph of $\Z^d$, consisting of edges $e$ for which $\omega(e)=1$, and the connected components of this random graph are called open clusters. 
It is well known (see \cite{BH}) that there exists a critical value $p_c\in(0,1)$ such that if $p>p_c$, an infinite open cluster exists almost surely, whereas no such cluster exists when $p< p_c$. We say that percolation occurs when an infinite open cluster exists.

Understanding connectivity in disordered systems remains a cornerstone of statistical physics and mathematics. Among the many percolation models, those incorporating environmental inhomogeneities offer fascinating insights into phase transitions in complex systems.  Perhaps the most influential work in this direction is that of McCoy and Wu \cite{MW}, in which the authors investigated disorder in the Ising model. Since then, the corresponding literature has expanded to include various other models, such as percolation \cite{Brochette,Marcos,H,JMP,KSV,Z}, the Ising model \cite{CK,CKP}, and the contact process \cite{A,BDS,K,L,NV}.  

An important work in this direction is the Brochette percolation model on $\Z^2$ introduced in \cite{Brochette}, which explores the effects of columnar disorder on connectivity in critical percolation on the square lattice. The authors show that even a small density of enhanced columns — where edges are more likely to connect — can enable percolation under specific conditions. 

In Brochette percolation on $\Z^2$, columns are chosen according to independent Bernoulli random variables with parameter $\rho$. Edges on enhanced columns are open with probability $q$, while the remaining edges are open with probability $p<q$. The authors of \cite{Brochette} showed that for all $\varepsilon>0$ and all $\rho>0$, there exists $\delta>0$ such that the model with parameters $p=p_c-\delta$ and $q=p_c+\varepsilon$ percolates, for almost all choices of the set of enhanced columns. Their proof requires a solid understanding of critical and near-critical percolation in two dimensions. In particular, the authors of \cite{Brochette} rely on a polynomial upper bound on the percolation correlation length (see \cite{Kesten1987} and \cite{Nolin}), combined with an ingenious use of Russo-Seymour-Welsh (RSW) techniques.  With these tools, the authors employ a one-step renormalization procedure, comparing the Brochette model to a certain oriented percolation process, which was shown to be supercritical in \cite{KSV}.

A natural and challenging problem is extending the Brochette model to dimension $d=3$. Additionally, it is worth exploring the scenario where columns are selected based on a heavy-tailed renewal process.
Vertical edges within selected columns are more likely to connect, creating a structured randomness reminiscent of real-world disordered systems. The primary difficulties in addressing this problem stem from the absence of  RSW-type estimates and the lack of polynomial upper bounds on the correlation length for Bernoulli percolation in dimensions $d\geq 3$. The best available bound, as established in \cite{DKT}, is exponential and therefore insufficient for this purpose. Moreover, the renormalization scheme proposed in \cite{KSV} is not suitable when the enhanced columns are determined by a heavy-tailed renewal process (see Theorem 1.2 in \cite{Marcos2}).

In this work, we study the Brochette model on slabs of the three-dimensional cubic lattice, where randomly enhanced columns are determined by a heavy-tailed renewal process. Compared with the usual setting of independent geometric gaps, this defines a broader class of spacing variables and introduces long-range dependencies on the environment. Edges along these selected columns are independently open with probability $
q$, while all other edges are independently open with probability $p$. It is known that percolation does not occur on slabs when both 
$p$ and $q$ are at the critical threshold (see \cite{DST}). Our main result shows that if 
$q$ is supercritical, then — even in the presence of predominantly critical (or even slightly subcritical) edges — the presence of these sparsely distributed “strong” columns is enough to ensure percolation.

Our slab-based approach bridges the known two-dimensional setting with the three-dimensional regime. In particular, we obtain a polynomial upper bound for the correlation length on slabs, which we hope can shed some light on the investigation of near-critical percolation on the three-dimensional cubic lattice and related models. 

As mentioned earlier, the multiscale argument in \cite{KSV} is not appropriate in our case. Hence, to show that percolation occurs, we take a different route and implement a renormalization scheme reminiscent of \cite{Marcos}. In that work, the authors study percolation on a square lattice where rows are randomly stretched according to a specific probability function. This creates heavy-tailed gaps between columns, resulting in columnar disorder with infinite-range vertical dependencies. Their multiscale scheme allows them to show a non-trivial phase transition for the model.

Two main difficulties arise when applying the ideas of \cite{Marcos}. First, we have to deal with dependence issues arising from a certain block construction. Second, in \cite{Marcos}, every horizontal edge is open with fixed probability $s$, regardless of the environment. This allows weak blocks in their multiscale scheme to be crossed along a straight line by choosing $s$ sufficiently large, without compromising the argument. In our context, however, horizontal edges have a high probability of being open only within strong blocks, while in weak blocks this probability decreases with the size of the blocks involved in the renormalization step. Consequently, a careful balance between the renormalization step and the multiscale argument is required to ensure that, even with a low density of open horizontal edges in some regions, the argument remains valid.

\subsection{Model and main results}\label{ModelandMainResult}

We study an inhomogeneous bond percolation process, described as follows. Denote the three-dimensional slab of thickness $k$ by $\S^+_k=(\mathcal{V}(\S^+_k),\mathcal{E}(\S^+_k))$,  with $\mathcal{V}(\S^+_k)=\mathbb{Z}^2_+\times\{0,...,k\}$ and $\mathcal{E}(\S^+_k) = \{\langle x,y \rangle: \sum_{i=1}^3|x_i-y_i|=1\}$. 

A vertical column indexed by $x\in\Z_+$ is the set $\{x\}\times\Z_+\times\{0\}$. Vertical columns are selected based on the arrivals of a renewal process with inter-arrival times whose tail distribution depends on a parameter $\phi>1$. Edges on selected vertical columns are independently open with probability $q$ and closed with probability $1-q$. All the remaining edges are independently open with probability $p$ and closed with probability $1-p$.

Compared to \cite{Brochette}, our model introduces two significant additional complexities:
\begin{enumerate}
\item We investigate percolation on slabs rather than $\Z^2$.  The shift from $\Z^2$ to slabs of $\Z^3$ introduces new significant challenges, as our understanding of near-critical behavior in higher dimensions remains incomplete. In particular, we derive a new upper bound for the correlation length of Bernoulli percolation on slabs.
\item We consider environments defined by a renewal process, instead of i.i.d. random variables. 
\end{enumerate}

A formal definition of our model reads as follows: let $\{\xi_i\}_{i\geq 1}$ be a sequence of independent copies of a random variable $\xi$ taking values in $\mathbb{N}=\{1,2,3,...\}$, with \begin{equation}\P(\xi>t)\leq c t^{-\phi},
\label{eq:alphaRenewal}
\end{equation}
for some $\phi>1$. Let $U$ be a random variable such that $$\P(U = k) = \frac{1}{\mathbb{E}(\xi)}\P(\xi>k),\,\,\,\,\,\,\,k\in\mathbb{N}.$$ 
Consider a stationary $\phi$-renewal process $\eta(U,\xi) = \{\eta_i\}_{i\geq 0}$ in $\mathbb{N}$, with inter-arrival times $\xi_i$ and delay time $U$, defined recursively as 
\begin{equation}\eta_0 = U,\quad \eta_i = \eta_{i-1}+\xi_i, \mbox{ for }i \in \mathbb{N}.
\label{renewalProcess}
\end{equation}

Percolation is introduced as follows: set the \first{environment} as $\Lambda = \{\eta_i:i \in \mathbb{N}\}$ and, for each $(p,q) \in [0,1]^2$, let $\P^\Lambda_{p,q}$ be the percolation measure on $\Omega=\{0,1\}^{\mathcal{E}(\S^+_k)}$ under which the $\omega(e)'s$ ($e \in \mathcal{E}(\S^+_k)$) are independent with 
$$\P^\Lambda_{p,q}(\omega(e)=1) = \left\{\begin{matrix}
p &\mbox{ if } e \notin  \mathcal{E}(\Lambda \times \mathbb{Z}_+\times\{0\}),\\
q &\mbox{ if } e \in  \mathcal{E}(\Lambda \times \mathbb{Z}_+\times\{0\}).
\end{matrix}\right.$$

Observe that $\P^{\Lambda}_{p,p}$ is the usual Bernoulli percolation measure with parameter $p$, which we denote by $\P_{p}$. 

Write $\{o\longleftrightarrow \infty\}$ for the event that the open connected component containing the origin is infinite (see Section \ref{notation} for a precise definition). Define the critical threshold as
$$p_c(\S^+_k)=\sup\{p:\P_{p}(o\longleftrightarrow \infty)=0\},$$
noting that $\P_{p_c}(o\longleftrightarrow \infty)=0$; see \cite{DST}.
For simplicity, whenever the context is clear, we will simply write $p_c$. We will prove the following theorem.

\begin{theorem} 
    \label{mainTheorem}
Consider the percolation model $\P_{p,q}^\Lambda$ on the slab $\S^+_k$, with $\Lambda$ given by an $\phi$-renewal process with law $\nu_{\phi}$. There exists $\phi_0 = \phi_0(k)$ such that, for every $\phi>\phi_0$ and every $\phi$-renewal process, the following holds: for all $\varepsilon>0$, there exists $\delta = \delta(\varepsilon)$ such that
$$\P^\Lambda_{p_c-\delta, p_c+\varepsilon}(o \longleftrightarrow \infty)>0, \quad \mbox{for $\nu_{\phi}$-almost every }\Lambda.$$
\end{theorem}

The argument leading to the proof of Theorem \ref{mainTheorem} is outlined in Section \ref{SectionOverview}. We perform a block construction, coupling our process with a 1-dependent two-dimensional percolation model. Following that, we show that this 1-dependent process dominates an independent model, which is shown to be supercritical. The final step employs a multi-scale argument. 

While our result holds only for sufficiently large $\phi$, we believe it fails when $\phi$ is below a threshold $\phi^*>1$, which could depend on the one-arm critical exponent. We stress that Theorem \ref{mainTheorem} also holds for i.i.d. geometric spacing variables, since \eqref{eq:alphaRenewal} is satisfied in this case.

As we mentioned before,  we derive a new upper bound on the correlation length for classical Bernoulli percolation on slabs, a key quantity for understanding phase transitions in physical complex systems. In fact, the entire scaling theory is fundamentally based on the idea that, for each $p\neq p_c$, there exists a single dominant length scale (known as the correlation length), and all quantities should be evaluated relative to this scale (see \cite{F} and \cite{S}).

Let $H([0,2n]\times[0,n]\times[0,k])$ denote the event that the box $[0,2n]\times[0,n]\times[0,k]$ is crossed horizontally.  For $\tau > 0 \mbox{ and } p > p_c$, define the correlation length at \(p\) by
\begin{equation}\label{cl_def} L_\tau(p) \coloneqq \inf\{n\geq 1: \P_p(H([0,2n]\times[0,n]\times[0,k]))\geq 1-\tau\}. 
\end{equation} 

We will prove the following result.

\begin{theorem}\label{corr_length_2}
Consider Bernoulli percolation on $\S_k^+$ with parameter $p>p_c$. For every $\tau>0$, there are constants $c_1 = c_1(\tau,k)>0$ and $c_2 =  c_2(\tau,k)>0$ such that 
    $$L_\tau(p)\leq c_1|p-p_c|^{-c_2}.$$
\end{theorem}

We observe that, by using RSW techniques (see  \cite{tassion} and the proof of Corollary 3.3 in \cite{CC}) together with the definition of $L_\tau(p)$, it follows that the crossing probability in \eqref{cl_def} remains above $1-\tau$ whenever $n>c_0L_{\tau}(p)$, for some constant $c_0>0$.

The proof of Theorem \ref{corr_length_2} involves showing that, with probability bounded away from zero, a box of side length $L_\tau(p)$ contains a vertex with three edge-disjoint open paths connecting it to the boundary of the box. One of these paths is separated from the other two by a closed cutset. To establish this, we combine the results from \cite{tassion} with localized modifications to specific configurations. This construction, together with the van den Berg, Kesten, and Reimer (BKR) inequality (see \cite{BK} and \cite{R}), provides a lower bound for the probability that an edge is pivotal to the crossing event of the box. Using this bound, we then apply the Margulis-Russo's formula to derive an estimate for the correlation length.

\subsection{Notation}\label{notation}
For each $A\subset \mathcal{V}(\S_k^+)$, let $\E(A)$ denote the set of edges with both endpoints in $A$. For $n\geq 1$, $A \subseteq \mathbb{Z}^2$, and $x\in\Z^2$, write $A_n(x) = 2nx+A+(n,n)$, and $\slab{A}_n(x)=A_n(x)\times\{0,...,k\}\subset \S_k$. For $x\in\mathbb{Z}^2$, let $B_n(x)=2nx+[-n,n]^2+(n,n)$ be the box of radius $n$ centered at $2nx+(n,n)$ in $\Z^2$. If $x=0$, we simply write $B_n$ instead of $B_n(0)$. 

An edge $e$ is said to be \first{open} if $\omega(e)=1$, otherwise it is said to be \first{closed}.  We write $u\sim v$ if $u$ and $v$ are neighbors, that is, if $\langle u,v \rangle\in\E(\S_k^+)$. A path of $(\mathcal{V}(\S_k^+), \E(\S_k^+))$ is an alternating sequence $x_0, e_0, x_1, e_1,\dots, e_{n-1},x_n$ of distinct vertices $x_i$ and edges $e_i=\langle x_i, x_{i+1}\rangle$; such a path has length $n$ and is said to connect $x_0$ to $x_n$. We call a path open if all of its edges are open. The event that $x_0$ is connected to $x_n$ by an open path is denoted by  $\{x_0 \longleftrightarrow x_n\}$. We write $\{x\longleftrightarrow \infty\}$ for the event that there exists an unbounded sequence $(x_n)\subset \mathcal{V}(\S_k^+)$ such that $x$ is connected by an open path to each one of the $x_n's$.

For each pair of paths $\gamma = (x_0,e_0,x_1,e_1,...,e_{n-1},x_n)$ and $\gamma' = (x'_0,e'_0,x'_1,...,e'_{n-1},x'_n)$ with $x_n=x'_0$, let $\gamma \oplus \gamma' = (x_0,e_0,...,x_n,e'_0,x'_1,...,e'_{n-1},x'_n)$ be the concatenation of the two paths. Denote by \begin{equation}\label{path_edgeset}
\mathfrak{E}(\gamma) = \{e_0,...,e_{n-1}\}
\end{equation}
the set of edges of the path $\gamma$.

For $A, B,C\subset \mathcal{V}(\S_k^+)$, we write $\{A\stackrel{C}{\longleftrightarrow}B\}$ for the event that some vertex of $A$ is connected to some vertex of $B$ using edges of $\E(C)$ only. Also, for each box $B = [a,b]\times[c,d]\times[0,k]$, let 
\begin{equation}\label{Hor_cross}
H(B)=\{\{a\}\times[c,d]\times[0,k]\stackrel{B}{\longleftrightarrow }\{b\}\times[c,d]\times[0,k]\},\end{equation}
\begin{equation}\label{ver_cross}
V(B)=\{[a,b]\times\{c\}\times[0,k]\stackrel{B}{\longleftrightarrow }[a,b]\times\{d\}\times[0,k]\},
\end{equation}
be the events where the box $B$ is crossed horizontally and vertically, respectively.

The distance between $u,v \in \mathcal{V}(S_k^+)$ is defined by $\|u-v\|=\sum_{i=1}^3|u_i-v_i|$. For every $A,B \subset \mathcal{V}(\S_k^+)$, write $d(A,B) = \min\{\|u-v\|: u\in A, v \in B\}$ for the distance between $A$ and $B$. For each set $B \subset \mathcal{V}(\S_k^+)$, let $$\frontA{B}:=\{u \in B: u\sim v \mbox{ for some } v\in \mathcal{V}(\S_k^+)\setminus B \}$$ 
be the vertex boundary of $B$.  
We abuse notation and also use $\frontA A$ for the vertex boundary of $A\subset\Z^2$, that is,
$$\frontA{A}:=\{u \in A: u\sim v \mbox{ for some } v\in \Z^2\setminus A \}.$$ 
\section{Proof overview and tools}\label{SectionOverview} 
In this section, we explain the main steps of the proof of Theorem \ref{mainTheorem}. Firstly, we characterize the notions of good and bad intervals that will be used below. Let $n\geq 1$, and define the intervals 
\begin{equation}\label{good_interval}
I^i_n = 2in+[0,2n),\,\, i \in \mathbb{Z}_+.
\end{equation}
Given an environment $\Lambda$ as in \eqref{renewalProcess}, let $r = |\Lambda\cap I^i_n|$ and consider the set $\Lambda\cap I^i_n=\{a_1,\dots,a_r\}$ with $a_j<a_{j+1}$, $j=1,\dots r-1$. Letting $a_0=2in$ and $a_{r+1}=2(i+1)n$, we define 
$$d_i = \max\left\{|a_j-a_{j-1}|: 1\leq j\leq r+1\right\}.$$

We say an interval $I^i_n$ is \first{$\lambda$-good} if $d_i\leq n^\lambda$, and \first{$\lambda$-bad} otherwise. For each fixed $\constDecaimento>1$, the following proposition shows that a given interval is $\lambda$-good with high probability for a suitable choice of $\lambda$, depending on $\constDecaimento$.
\begin{proposition}
\label{goodBlocksProposition}
Let $\Lambda$ be given by an $\constDecaimento$-renewal process with law $\nu_\constDecaimento$. If $\constDecaimento>1$ and $\lambda \in \left(\frac{1}{\constDecaimento},1\right)$, then
$$\nu_\constDecaimento(I^i_n \mbox{ is }\lambda\mbox{-good})\geq 1-3cn^{1-\constDecaimento \lambda},$$
for all $i \in \mathbb{Z}_+$ and all $n \in \mathbb{N}$.
\end{proposition}
\begin{proof}

It holds that $d_i>n^\lambda$ if and only if, there is $j \in \mathbb{Z}_+$ with $\eta_j \in [2in,2(i+1)n-n^\lambda)$ and $\xi_j>n^\lambda$ or with $\eta_j<2in$ and $\xi_j>|2in-\eta_j|+n^\lambda$. If $\ell\in \Lambda$, then $\ell=\eta_j$ for some $j\in\Z_+$, in which case we denote $\xi_{[\ell]}:=\xi_{j+1}$. An application of the union bound gives 
\begin{align*}
    \nu_\constDecaimento(I^i_n \mbox{ is }\lambda\mbox{-bad})&\leq
         \sum_{\ell=0}^{2in-1}\nu_\constDecaimento\left(\ell \in \Lambda,\,  \xi_{[\ell]}>2in-\ell+n^\lambda
        \right)+\sum_{\ell=2in}^{2(i+1)n}\nu_\constDecaimento\left(\ell \in \Lambda,\, \xi_{[\ell]}>n^\lambda\right)\\
       &\leq \sum_{\ell=0}^{2in-1}\nu_\constDecaimento\left(\xi>2in-\ell+n^\lambda\right)+\sum_{\ell=2in}^{2(i+1)n}\nu_\constDecaimento\left(\xi>n^\lambda\right)\\
       &\leq \sum_{\ell=n^\lambda}^{\infty}\nu_\constDecaimento\left(\xi>\ell\right)+cn^{1-\constDecaimento \lambda}
       \leq \sum_{\ell=n^\lambda}^{\infty}c\ell^{-\constDecaimento}+ cn^{1-\constDecaimento \lambda}\leq 3cn^{1-\constDecaimento \lambda}.
\end{align*}
\end{proof}
The first part of the proof of Theorem \ref{mainTheorem} consists of a coarse-graining argument, coupling the states of the blocks involved with a dependent percolation model in $\mathbb{Z}^2_+$. We begin by defining the sets and events that will be used in our construction.

 \begin{figure}[ht]
        \centering
        \resizebox{0.4\linewidth}{!}{
            \begin{tikzpicture}


                \draw[color=red!0,pattern=north west lines, pattern color=black!75] (-6.3,-3) rectangle (-6,3);
                \draw[color=red!0,pattern=north west lines, pattern color=black!75] (6,-3) rectangle (6.3,3);
                \draw[color=red!0,pattern=north west lines, pattern color=black!75] (-3,-6.3) rectangle (3,-6);
                \draw[color=red!0,pattern=north west lines, pattern color=black!75] (-3,6) rectangle (3,6.3);

                \draw[color=black, ultra thick] (-6,-6) rectangle (6,6);
                \draw[color=black,pattern=north west lines, pattern color=black!25, ultra thick] (-6,-6) rectangle (-3,-3);
                \draw[color=black,pattern=north west lines, pattern color=black!25, ultra thick] (3,3) rectangle (6,6);

                \node[align=left,color=black,scale=3] at (-7.5,0) {$LS_n$};
                \node[align=left,color=black,scale=3] at (7.5,0) {$RS_n$};
                \node[align=left,color=black,scale=3] at (0,7) {$TS_n$};
                \node[align=left,color=black,scale=3] at (0,-7) {$BS_n$};

                \node[align=center,color=black,scale=3] at (-4.5,-4.5){$LB_n$};
                \node[align=center,color=black,scale=3] at (4.5,4.5){$RB_n$};

                \node[align=center,color=black,scale=3] at (-3,3){$B_n$};

                \filldraw [black] (0,0) circle (2pt);
                \node[align=left,color=black,scale=2] at (0.5,0.5) {$0$};                

                \draw[color=black] (-6.3,3) -- (-5.7,3);
                \node[align=left,color=black,scale=2] at (-7,3) {$\frac{n}{2}$};                
                \draw[color=black] (-6.3,-3) -- (-5.7,-3);
                \node[align=left,color=black,scale=2] at (-7,-3) {$-\frac{n}{2}$};                

                \draw[color=black] (6.3,3) -- (5.7,3);
                \node[align=left,color=black,scale=2] at (7,3) {$\frac{n}{2}$};                
                \draw[color=black] (6.3,-3) -- (5.7,-3);
                \node[align=left,color=black,scale=2] at (7,-3) {$-\frac{n}{2}$};

                \draw[color=black] (-3,-6.3) -- (-3,-5.7);
                \node[align=left,color=black,scale=2] at (-3,-7) {$-\frac{n}{2}$};                
                \draw[color=black] (3,-6.3) -- (3,-5.7);
                \node[align=left,color=black,scale=2] at (3,-7) {$\frac{n}{2}$};

                \draw[color=black] (-3,6.3) -- (-3,5.7);
                \node[align=left,color=black,scale=2] at (-3,7) {$-\frac{n}{2}$};                
                \draw[color=black] (3,6.3) -- (3,5.7);
                \node[align=left,color=black,scale=2] at (3,7) {$\frac{n}{2}$};

                \draw[color=black] (-6,-6.3) -- (-6,-5.7);
                \node[align=left,color=black,scale=2] at (-6,-7) {$-n$};                
                \draw[color=black] (6,-6.3) -- (6,-5.7);
                \node[align=left,color=black,scale=2] at (6,-7) {$n$};

                \draw[color=black] (-6.3,6) -- (-5.7,6);
                \node[align=left,color=black,scale=2] at (-7,6) {$n$};
            \end{tikzpicture}
        }
        \caption{Sketch of the sets $B_n,LB_n,RB_n, LS_n,RS_n,BS_n$ and $TS_n$.}
        \label{fig:boxes}
    \end{figure}
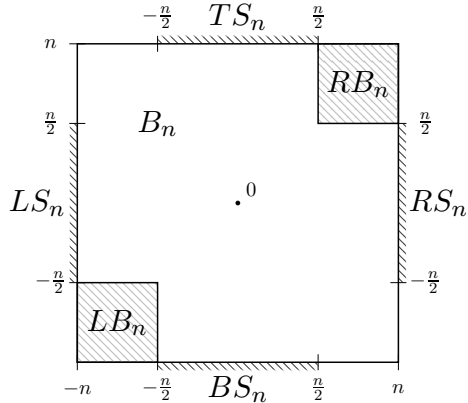 
\FloatBarrier

 Consider the following subsets of $B_n$: $LB_n = \left[-n,-\frac{n}{2}\right]^2$, and  $RB_n = \left[\frac{n}{2},n\right]^2$. Also, write $LS_n = \{-n\}\times\left[-\frac{n}{2},\frac{n}{2}\right]$, $RS_n = \{n\}\times\left[-\frac{n}{2},\frac{n}{2}\right]$, $BS_n = \left[-\frac{n}{2},\frac{n}{2}\right]\times \{-n\}$, and $TS_n =\left[-\frac{n}{2},\frac{n}{2}\right]\times\{n\}$. See Figure \ref{fig:boxes} for a sketch of these sets.
   
Define also the rectangles $HR_n = [-n,3n]\times\left[-\frac{n}{2},\frac{n}{2}\right]$ and  $VR_n=\left[-\frac{n}{2},\frac{n}{2}\right]\times[-n,3n]$. Finally, write
\begin{align*}
D_n(x)&\coloneqq \left\{\mathbf{LB}_n(x) \overset{\mathbf{B}_n(x)}{\looongleftrightarrow} \mathbf{RB}_n(x)\right\} ,\\
H_n(x)&\coloneqq \left\{\slab{LS}_n(x) \overset{\slab{HR}_n(x)}{\looongleftrightarrow} \slab{RS}_n(x+(1,0))\right\},\\
V_n(x)&\coloneqq \left\{\slab{BS}_n(x) \overset{\slab{VR}_n(x)}{\looongleftrightarrow} \slab{TS}_n(x+(0,1))\right\}.     
\end{align*}
See Figure \ref{fig:events} for a two-dimensional sketch of these events.
        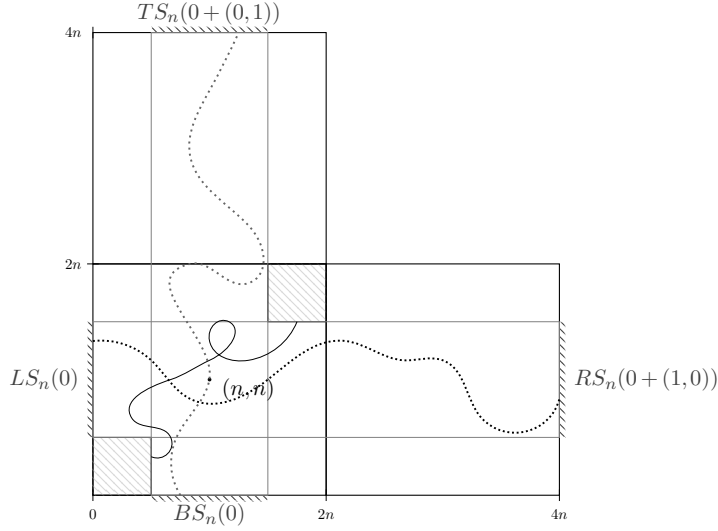
\begin{figure}[ht]
        \centering
        \resizebox{0.6\linewidth}{!}{
            \begin{tikzpicture}


                \draw[color=red!0,pattern=north west lines, pattern color=black!75] (-6.3,-3) rectangle (-6,3);
                \draw[color=red!0,pattern=north west lines, pattern color=black!75] (18,-3) rectangle (18.3,3);
                \draw[color=red!0,pattern=north west lines, pattern color=black!75] (-3,-6.3) rectangle (3,-6);
                \draw[color=red!0,pattern=north west lines, pattern color=black!75] (-3,18) rectangle (3,18.3);

                \draw[color=black, ultra thick] (-6,-6) rectangle (6,6);
                \draw[color=black,pattern=north west lines, pattern color=black!15, ultra thick] (-6,-6) rectangle (-3,-3);
                \draw[color=black,pattern=north west lines, pattern color=black!15, ultra thick] (3,3) rectangle (6,6);
                
                \draw[color=black, ultra thick] (-6,6) rectangle (6,18);
                \draw[color=black, ultra thick] (6,-6) rectangle (18,6);

                \draw[color=black!50] (-3,-6) rectangle (3,18);
                \draw[color=black!50] (-6,-3) rectangle (18,3);

                \node[align=left,color=black!75,scale=3] at (-8.5,0) {$LS_n(0)$};
                \node[align=left,color=black!75,scale=3] at (22.5,0) {$RS_n(0+(1,0))$};
                \node[align=left,color=black!75,scale=3] at (0,19) {$TS_n(0+(0,1))$};
                \node[align=left,color=black!75,scale=3] at (0,-7) {$BS_n(0)$};

                \filldraw [black] (0,0) circle (2pt);
                \node[align=left,color=black,scale=3] at (2,-0.5) {$(n,n)$};                

                \draw[color=black] (-6,-6.3) -- (-6,-5.7);
                \node[align=left,color=black,scale=2] at (-6,-7) {$0$};                
                \draw[color=black] (6,-6.3) -- (6,-5.7);
                \node[align=left,color=black,scale=2] at (6,-7) {$2n$}; 
                \draw[color=black] (18,-6.3) -- (18,-5.7);
                \node[align=left,color=black,scale=2] at (18,-7) {$4n$};
                \draw[color=black] (-6,18.3) -- (-6,17.7);
                \node[align=left,color=black,scale=2] at (-7,18) {$4n$};
                
                \draw[color=black] (-6.3,6) -- (-5.7,6);
                \node[align=left,color=black,scale=2] at (-7,6) {$2n$};
                
                \draw [line width=1mm,color=black!60, loosely dashed] (-1.5 ,-6) to [ curve through ={(-1.7,-3).. (0,0.5)..(-1,6) . . (2,5) . . (2,8) . . (-1,11.6) . . (0,15)  }] (1.5,18);

                \draw [line width=1mm,color=black,loosely dotted] (-6 ,2) to [ curve through ={(-3,1).. (-1,-1)..(3,0)..(7,2)..(10,1)..(12,1)..(14,-2)}] (18,-1);

                \draw [very thick,color=black] (-3,-4) to [ curve through ={(-2,-3)..(-4,-2)..(-2,0)..(0,1)..(1,3)..(0,2)..(2,1)}] (4.5,3);

            \end{tikzpicture}
        }
        \caption{Two-dimensional sketch of $D_n(0)$ (continuous line), $H_n(0)$ (dotted dark line) and $V_n(0)$ (dashed light line).}
        \label{fig:events}
    \end{figure} 
    \FloatBarrier

Let $\mathcal{J}_n=\{i \in \mathbb{Z}_+: I^i_n \mbox{ is }\lambda\mbox{-good}\}$,  $\mathcal{U}_n = \bigcup_{i \in \mathcal{J}_n} I^i_n$, and write $U_n = \mathcal{U}_n \times \mathbb{Z}$. We shall associate to every edge $\langle x,y \rangle \in \mathcal{E}(\mathbb{Z}_+^2)$ the set $\slab{B}_n(x)\cup \slab{B}_n(y)\subseteq \S^+_k$. We say $\langle x,y \rangle \in \mathcal{E}(\mathbb{Z}_+^2)$ is $\lambda$-\first{favored} if $B_n(x)\cup B_n(y) \subseteq U_n$. Otherwise, we say $\langle x,y \rangle$ is $\lambda$-\first{unfavored}. Hereafter, when there is no possibility of confusion, we will omit $\lambda$ from the notation.  

\begin{definition}\label{pathOrdering}Let $\prec$ be the lexicographical ordering of the vertices of $\S^+_k$. We introduce an ordering $\leq$ on the set of self-avoiding paths in $\S_k^+$ as follows. Given two self-avoiding paths $\gamma = (\gamma_0,...,\gamma_s)$ and $\gamma' = (\gamma'_0,...,\gamma'_t)$ in $\S_k^+$, we say $\gamma\leq \gamma'$ if at least one of the following conditions holds:
\begin{enumerate}
\label{ordering}
    \item $s<t$ and $\gamma_j = \gamma'_j$, for all $j \leq s$,
    \item $\gamma_0\prec \gamma'_0$,
    \item there exists $k < \min\{s,t\}\mbox{ such that } \gamma_j = \gamma'_j$, for all $j \leq k$, and $\gamma_{k+1}\prec\gamma'_{k+1}$.
\end{enumerate}
\end{definition}     

For each vertex $x \in \mathbb{Z}^2_+$, let $\Xi_n(x)$ be the set of all self-avoiding paths joining $\slab{LB}_n(x)$ to $\slab{RB}_n(x)$ inside $\slab{B}_n(x)$. If $\omega \in D_n(x)$, let
\begin{equation}\label{DefinitionGamma}\Gamma_x^n(\omega)\coloneqq \min\{\gamma \in \Xi_n(x): \gamma \mbox{ open}\},\end{equation}
that is, the smallest open path in $\Xi_n(x)$. We write $\Gamma_x^n(\omega) = \emptyset$ when $\omega \notin D_n(x)$. 

For every edge $f=\langle x,y \rangle$ in $\mathcal{E}(\mathbb{Z}^2_+)$, let $R(f) = HR_n(x)$ if $f = \langle x, x+(1,0) \rangle$, and $R(f) = VR_n(x)$ if $f = \langle x, x+(0,1) \rangle$. Define also the set of ``corner edges" (see Figure \ref{fig:C(e)})  
\begin{small}
    \begin{equation}\label{front}C(f) \coloneqq \E\left[\frontA{\slab{LB}_n(x)}\cup\frontA{\slab{RB}_n(x)}\cup\frontA{\slab{LB}_n(y)}\cup\frontA{\slab{RB}_n(y)}\right]\setminus\E\left[\frontA{\slab{B}_n(x)}\cup \frontA{\slab{B}_n(y)}\right].
\end{equation}
\end{small}

    \begin{figure}[ht]
        \centering
        \resizebox{0.6\linewidth}{!}{
            \begin{tikzpicture}
             \draw[very thick, black!15](4,0) rectangle (20,2);
             \draw[very thick, black!15](12,0) -- (12,2);             
             \draw[very thick, black!15](0,4) rectangle (16,6);
             \draw[very thick, black!15](8,4) -- (8,6);

             \draw[very thick, black!15] (4,0)--(0,4);
             \draw[very thick, black!15] (4,2)--(0,6);
             \draw[very thick, black!15] (12,0)--(8,4);
             \draw[very thick, black!15] (12,2)--(8,6);
             \draw[very thick, black!15] (16,4)--(20,0);
             \draw[very thick, black!15] (20,2)--(16,6);            

            \draw[very thick, black!0,pattern=grid,opacity=0.5] (3,1) rectangle (5,3);
            \draw[very thick, black!0,pattern=north west lines,opacity=0.3](5,3) --(6,2) --(6,0) --(5,1) -- cycle;
            \draw[very thick, black!0,pattern=vertical lines,opacity=0.3](5,3) --(6,2) --(6,0) --(5,1) -- cycle;

            \draw[very thick, black!0,pattern=grid,opacity=0.5] (11,1) rectangle (13,3);
            \draw[very thick, black!0,pattern=north west lines,opacity=0.3](13,3) --(14,2) --(14,0) --(13,1) -- cycle;
            \draw[very thick, black!0,pattern=vertical lines,opacity=0.3](13,3) --(14,2) --(14,0) --(13,1) -- cycle;

            \draw[very thick, black!0,pattern=grid,opacity=0.5] (7,3) rectangle (9,5);
            \draw[very thick, black!0,pattern=north west lines,opacity=0.3](7,3) --(7,5) --(6,6) --(6,4) -- cycle;
            \draw[very thick, black!0,pattern=vertical lines,opacity=0.3](7,3) --(7,5) --(6,6) --(6,4) -- cycle;

            \draw[very thick, black!0,pattern=grid,opacity=0.5] (15,3) rectangle (17,5);
            \draw[very thick, black!0,pattern=north west lines,opacity=0.3](15,3) --(15,5) --(14,6) --(14,4) -- cycle;
            \draw[very thick, black!0,pattern=vertical lines,opacity=0.3](15,3) --(15,5) --(14,6) --(14,4) -- cycle;

            \node[scale=2] at (9,0.5) {$\slab{B}_n(x)$};
            \node[scale=2] at (17,0.5) {$\slab{B}_n(y)$};
            \end{tikzpicture}
        }
        \caption{A sketch of the sets $\slab{B}_n$ and $C(f)$.}
        \label{fig:C(e)}
    \end{figure}
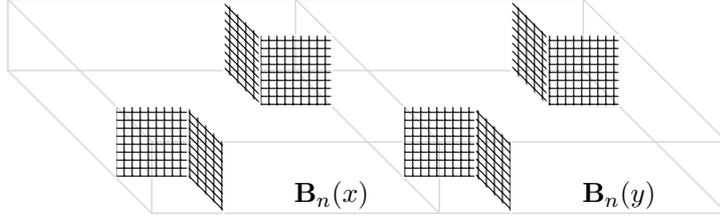    
For each $f=\langle x,y \rangle \in \mathcal{E}(\mathbb{Z}^2_+)$, define 
\begin{equation}\label{DefinitionA}A_n(f) \coloneqq D_n(x)\cap D_n(y)\cap \left\{\Gamma_x^n(\omega) \overset{\slab{R}(f)}{\looongleftrightarrow} \Gamma_y^n(\omega)\right\}.\end{equation}
One can think of a renormalized favored edge $f$ as being open whenever the event $A_n(f)$ occurs. The next proposition states that the probability of the event $A_n(f)$ can be made arbitrarily high provided $n$ is large enough and the parameters are chosen appropriately. 

\begin{proposition}
    \label{theoremEdge}
For each $\constNossoBeta<1$, there exists $\lambda=\lambda(\constNossoBeta)>0$ such that the following holds: for every $\varepsilon>0$, there exists $\delta_n = \delta(n, \varepsilon, \lambda)>0$ such that, for $n$ large enough,
        \begin{equation}\label{eq:Prop2Equation}\P^\Lambda_{p_c-\delta_n,p_c+\varepsilon}(A_n(f)) \geq \constNossoBeta, \mbox{ if $f$ is }\lambda\mbox{-\textit{favored}}.\end{equation}
\end{proposition}
\begin{remark}
    \label{RemarkEscolhaBeta}
    To prove Theorem \ref{mainTheorem}, we will take $\constNossoBeta_0$ as in \eqref{DefBeta0}, $\lambda=\lambda(\constNossoBeta_0,k)$ and let $\constDecaimento_0(k)>1/{\lambda}$. 
\end{remark}

The multiscale scheme we describe below requires that favored edges are open with high probability, while merely requiring that unfavored edges are open with positive probability. With this in mind, we say an unfavored edge is open if
\begin{equation}\label{unfav_open}
A_n^*(f) \coloneqq A_n(f)\cap \{\omega(e) = 1 \mbox{ for all } e \in C(f)\}
\end{equation}
occurs. Note that the event in \eqref{unfav_open} is increasing, allowing the use of the FKG inequality.

The proof of Proposition \ref{theoremEdge}, which is postponed to Section \ref{boundsEdge}, builds on the approach used in Proposition 2 of \cite{Brochette}. The most complex aspect of the proof involves a thorough understanding of near-critical percolation, which is challenging due to the current lack of such understanding on slabs. A vital part of the proof of Proposition 2 in \cite{Brochette} involves comparing the crossing probabilities of ``good'' renormalized edges with those in a homogeneous process at parameter $p_n = p_c(\Z^2)+q_n$, where $q_n$ approaches zero. For slabs, the desired outcome (with $p_c(\S^+_k)$ instead of $p_c(\Z^2)$) can be achieved for favored edges through successive applications of the Gluing Lemma, as detailed in Lemma \ref{GluingLemmaGenerico}. This is complemented by the Aizenman-Grimmett method to compare pivotality probabilities (e.g., \cite{AizenmannGrimmett} and \cite{Brochette}), which requires a highly nontrivial adaptation to our specific context.

The next step is to show that we are operating at length scales beyond the correlation length for $p_n$. Theorem \ref{corr_length_2} provides an upper bound for $L_{\tau}(p)$, and Section 3 is dedicated solely to its proof.

Once Theorem \ref{corr_length_2} is established, we will prove that $n\geq L_\tau(p_n)$ for sufficiently large $n$, thereby indicating that the probabilities of the events $D_n$, $H_n$, and $V_n$ can be made arbitrarily close to one. Consequently, by applying the Gluing Lemma, we will be able to prove Proposition \ref{theoremEdge}.

The construction with block variables $A_n(f)$ and $A_n^*(f)$ induces a dependent percolation model as follows: to every percolation configuration $\omega \in \{0,1\}^{\mathcal{E}(\S^+_k)}$ associate a configuration $\sigma_n(\omega) \in \{0,1\}^{\mathcal{E}(\mathbb{Z}_+^2)}$ according to the following rule:
\begin{equation}\label{BlockConstruction}
    \sigma_n(\omega)(f) = \left\{\begin{array}{ll}
        1 &\mbox{if }f\mbox{ is favored and } \omega \in A_n(f),\\
        1 &\mbox{if }f\mbox{ is unfavored and } \omega \in A_n^*(f),\\
        0 &\mbox{otherwise.}
    \end{array}\right.
\end{equation}

Let $\mathbb{Q}^{\Lambda,n}_{p,q} = \P^\Lambda_{p,q} \circ \sigma_n^{-1}$ denote the percolation measure in $(\mathbb{Z}_+^2, \mathcal{E}(\mathbb{Z}_+^2))$ induced by $\sigma_n$. Observe that $\mathbb{Q}^{\Lambda,n}_{p,q}$ is 1-dependent and, if $\sigma_n(\omega) \in \{o\longleftrightarrow \infty\}$, then $\omega \in \{o \longleftrightarrow \infty\}$ (see Figure \ref{fig:RenormalizacaoTotal}). Therefore, it suffices to show that 
\begin{equation}
    \inf_{m}\{\mathbb{Q}^{\Lambda,n}_{p_c-\delta,p_c+\varepsilon}(o\longleftrightarrow \frontA B_m)\}>0.
    \label{endMultiscale}
\end{equation}
    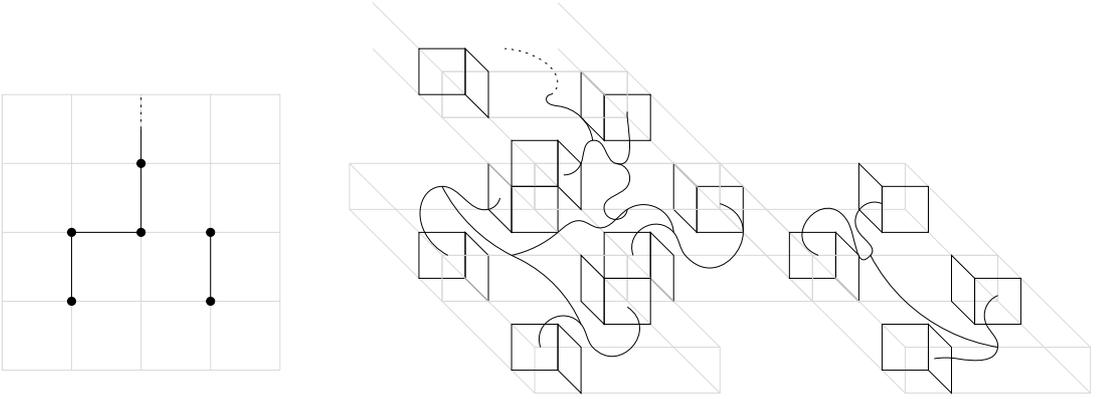
\begin{figure}[ht]
        \centering
        \resizebox{0.9\linewidth}{!}{
            \begin{tikzpicture}
                \draw[color=black!15, step=3](-15,-3) grid (-3,9);
                \filldraw[color=black] (-12,0) circle (5pt);
                \filldraw[color=black] (-12,3) circle (5pt);
                \filldraw[color=black] (-9,3) circle (5pt);
                \filldraw[color=black] (-9,6) circle (5pt);
                \filldraw[color=black] (-6,3) circle (5pt);
                \filldraw[color=black] (-6,0) circle (5pt);

                \draw[black](-12,0) -- (-12,3)--(-9,3)--(-9,6);
                \draw[black](-6,0) -- (-6,3);
                \draw[black](-9,6)--(-9,7.5);
                \draw[black,loosely dashed](-9,7.5)--(-9,9);

            \draw[very thick, black!15](4,0) rectangle (12,2);
            \draw[very thick, black!15](0,4) rectangle (8,6);
            \draw[very thick, black!15] (4,0)--(0,4);
            \draw[very thick, black!15] (4,2)--(0,6);
            \draw[very thick, black!15] (12,0)--(8,4);
            \draw[very thick, black!15] (12,2)--(8,6);
            \draw[very thick,opacity=0.2] (3,1) rectangle (5,3);
            \draw[very thick, opacity=0.2](5,3) --(6,2) --(6,0) --(5,1) -- cycle;
            \draw[very thick, black, opacity=0.2] (7,3) rectangle (9,5);
            \draw[very thick, black, opacity=0.2](7,3) --(7,5) --(6,6) --(6,4) -- cycle;

            \draw[very thick, black!15](12,0) rectangle (20,2);
            \draw[very thick, black!15](8,4) rectangle (16,6);
            \draw[very thick, black!15] (12,0)--(8,4);
            \draw[very thick, black!15] (12,2)--(8,6);
            \draw[very thick, black!15] (16,4)--(20,0);
            \draw[very thick, black!15] (20,2)--(16,6);                        
            \draw[very thick, black, opacity=0.2] (11,1) rectangle (13,3);
            \draw[very thick, black, opacity=0.2](13,3) --(14,2) --(14,0) --(13,1) -- cycle;
            \draw[very thick, black, opacity=0.2] (15,3) rectangle (17,5);
            \draw[very thick, black, opacity=0.2](15,3) --(15,5) --(14,6) --(14,4) -- cycle;

            \draw[very thick, black!15](8,4) rectangle (16,6);
            \draw[very thick, black!15](4,8) rectangle (12,10);
            \draw[very thick, black!15] (8,4)--(4,8);
            \draw[very thick, black!15] (8,6)--(4,10);
            \draw[very thick, black!15] (12,8)--(16,4);
            \draw[very thick, black!15] (16,6)--(12,10);
            \draw[very thick, black, opacity=0.2] (7,5) rectangle (9,7);
            \draw[very thick, black, opacity=0.2](9,7) --(10,6) --(10,4) --(9,5) -- cycle;
            \draw[very thick, black, opacity=0.2] (11,7) rectangle (13,9);
            \draw[very thick, black, opacity=0.2](11,7) --(11,9) --(10,10) --(10,8) -- cycle;

            \draw[very thick, black!15](8,-4) rectangle (16,-2);
            \draw[very thick, black!15](4,0) rectangle (12,2);
            \draw[very thick, black!15] (8,-4)--(4,0);
            \draw[very thick, black!15] (8,-2)--(4,2);
            \draw[very thick, black!15] (16,-4)--(12,0);
            \draw[very thick, black!15] (16,-2)--(12,2);
            \draw[very thick,opacity=0.2] (7,-3) rectangle (9,-1);
            \draw[very thick, opacity=0.2](9,-1) --(10,-2) --(10,-4) --(9,-3) -- cycle;
            \draw[very thick, black, opacity=0.2] (11,-1) rectangle (13,1);
            \draw[very thick, black, opacity=0.2](11,-1) --(11,1) --(10,2) --(10,0) -- cycle;

            \draw[very thick, black!15](20,0) rectangle (28,2);
            \draw[very thick, black!15](16,4) rectangle (24,6);
            \draw[very thick, black!15] (20,0)--(16,4);
            \draw[very thick, black!15] (20,2)--(16,6);
            \draw[very thick, black!15] (24,4)--(28,0);
            \draw[very thick, black!15] (28,2)--(24,6);                        
            \draw[very thick, black, opacity=0.2] (19,1) rectangle (21,3);
            \draw[very thick, black, opacity=0.2](21,3) --(22,2) --(22,0) --(21,1) -- cycle;
            \draw[very thick, black, opacity=0.2] (23,3) rectangle (25,5);
            \draw[very thick, black, opacity=0.2](23,3) --(23,5) --(22,6) --(22,4) -- cycle;

            \draw[very thick, black!15](24,-4) rectangle (32,-2);
            \draw[very thick, black!15](20,0) rectangle (28,2);
            \draw[very thick, black!15] (24,-4)--(20,0);
            \draw[very thick, black!15] (24,-2)--(20,2);
            \draw[very thick, black!15] (28,0)--(32,-4);
            \draw[very thick, black!15] (32,-2)--(28,2);
            \draw[very thick, black, opacity=0.2] (23,-3) rectangle (25,-1);
            \draw[very thick, black, opacity=0.2](25,-1) --(26,-2) --(26,-4) --(25,-3) -- cycle;
            \draw[very thick, black, opacity=0.2] (27,-1) rectangle (29,1);
            \draw[very thick, black, opacity=0.2](27,-1) --(27,1) --(26,2) --(26,0) -- cycle;

            \draw[very thick, black!15](4,8) rectangle (12,10);
            \draw[very thick, black!15] (4,8)--(1,11);
            \draw[very thick, black!15] (4,10)--(1,13);
            \draw[very thick, black!15] (9,11)--(12,8);
            \draw[very thick, black!15] (12,10)--(9,13);
            \draw[very thick, black, opacity=0.2] (3,9) rectangle (5,11);
            \draw[very thick, black, opacity=0.2](5,11) --(6,10) --(6,8) --(5,9) -- cycle;

            \draw [very thick,color=black] (8.25,-2) to [ curve through ={(10,-1)..(10.5,-2)}] (12,-0.25);
            \draw [very thick,color=black] (4.25,2) to [ curve through ={(3.25,3)..(4,5)..(6,4)}] (6.5,4.5);
            \draw [very thick,color=black] (12.25,2) to [ curve through ={(14,3)..(14.5,2)}] (16,4.25);
            \draw [very thick,color=black] (20.25,2) to [ curve through ={(21.5,3.5)..(22.5,2)..(22,3)}] (23,4.25);
            \draw [very thick,color=black] (25.25,-2.5) to [ curve through ={(26.25,-2.5)..(28,-2)..(27.5,-1)}] (28,0.25);
            \draw [very thick,color=black] (9.25,5.5) to [ curve through ={(10,6)..(10.5,7)..(11.5,6)..(12,8)}] (12,8.25);

            \draw [very thick,color=black] (10,-1) to [ curve through ={(7.4,1.8)..(7,2)}] (4,5);
            \draw [very thick,color=black] (7,2) to [ curve through ={(9,3)..(10,3.5)..(10.75,3.2)..(12,4)}] (14,3);
            \draw [very thick,color=black] (12,4) to [ curve through ={(11,4)..(12,5)}] (11.5,6);
            \draw [very thick,color=black] (28,-2) to [ curve through ={(24,0)}] (22.5,2);
            \draw [very thick,color=black] (10.5,7) to [ curve through ={(9,8.5)..(8.5,8.8)}] (8.7,9);
            \draw [very thick,color=black,loosely dashed] (8.7,9) to [ curve through ={(8.9,9.2)(8.3,10.5)}] (6.7,11);

            \end{tikzpicture}
        }
        \caption{The renormalized percolation and the original equivalent.}
        \label{fig:RenormalizacaoTotal}
    \end{figure}    

To prove \eqref{endMultiscale}, we will first show that, restricted to each finite box $B_m$, a randomized version of $(\sigma_n(f))_{f\in\E(\Z_+^2)}$ stochastically dominates an independent Bernoulli percolation model on $\mathbb{Z}_+^2$. This is the content of the next proposition.
\begin{proposition}\label{TeoDominacao}
There is a function $g:[0,1]\to[0,1]$, with $\lim_{x \to 1}g(x)=1$, and a constant $\constNossoBeta_1>0$, such that the following holds: 
let $0<\constNossoMu<\tfrac{1}{2}$, $m\geq 1$, $\{Z_f\}_{f \in \mathcal{E}(B_m)}\overset{i.i.d.}{\sim} Ber(1-\constNossoMu)$, and $W_f = \sigma_n(f)Z_f$. For all $\constNossoBeta>\constNossoBeta_1$, if
$$\mathbb{Q}^{\Lambda,n}_{p,q}(W_f=1)>\constNossoBeta,\; \mbox{ for }f\mbox{ favored},$$
then the random variables $\{W_f\}_{f \in \mathcal{E}(B_m)}$ dominate stochastically a family of independent random variables $\{Y_f\}_{f \in \mathcal{E}(B_m)}$ such that
$$    
\mathbb{Q}^{\Lambda,n}_{p,q}(Y_f=1) = \left\{\begin{array}{ll}
    g(\constNossoBeta) &\mbox{ if }f\mbox{ is favored},\\
    c_3(1-\constNossoMu)\P^{\Lambda}_{p,q}(\sigma_n(f)=1) &\mbox{ otherwise},
\end{array}\right.
$$
for some constant $c_3 = c_3(n)>0$.
\end{proposition}
To conclude the proof of Theorem \ref{mainTheorem}, we implement a multiscale argument based on \cite{Marcos} to show that the process $\{Y_f\}_{f\in \E(\Z_+^2)}$ percolates. Once this is established, we apply Proposition \ref{TeoDominacao} to show that the sequence $\{Y_f\}_{f\in \E(\Z_+^2)}$ is stochastically dominated by $\{W_f\}_{f\in \E(\Z_+^2)}$, which, in turn, is clearly dominated by $\{\sigma_n(f)\}_{f\in \E(\Z_+^2)}$. This will allow us to conclude that the original process percolates.

The remainder of this paper is organized as follows: in Section \ref{cl_proof}, we derive several auxiliary results leading to the proof of Theorem \ref{corr_length_2}. Specifically, we derive a lower bound for the probability that an edge is pivotal for the event $V(\slab{B}_N)$ (see \eqref{ver_cross}) within the homogeneous model. Section \ref{boundsEdge} is dedicated to proving Proposition \ref{theoremEdge}. To accomplish this, we introduce a version of the Gluing Lemma for lowest paths and use the Aizenman-Grimmett \cite{AizenmannGrimmett} approach to establish a slightly supercritical homogeneous process that dominates our model. Finally, in Section \ref{secMulti}, we show that our renormalized model dominates an independent Bernoulli percolation process on $\Z_+^2$. We then use a multiscale argument to complete the proof of Theorem \ref{mainTheorem}.

\section{An upper bound for the correlation length on slabs}\label{cl_proof}

\subsection{Proof of Theorem \ref{corr_length_2}}
In this section, we consider Bernoulli bond percolation on the slab $\S^+_k$ with parameter $p>p_c(\S^+_k)$. All constants appearing in this section depend on $k$, and we omit this dependence from the notation. Before we proceed to the proof of Theorem \ref{corr_length_2}, we need to introduce further notation.

Consider the box $\slab{B}_N$ and let $R_N' = [-N,N]\times[-N,0]\times[0,k]$, $R_N''=\left[-\frac{N}{2},\frac{N}{2}\right]\times[-N,N]\times[0,k]$, and $Q_N = [-N,N]\times\left[\frac{4N}{6},\frac{5N}{6}\right]\times[0,k]$; see Figure \ref{fig:crossings}.
    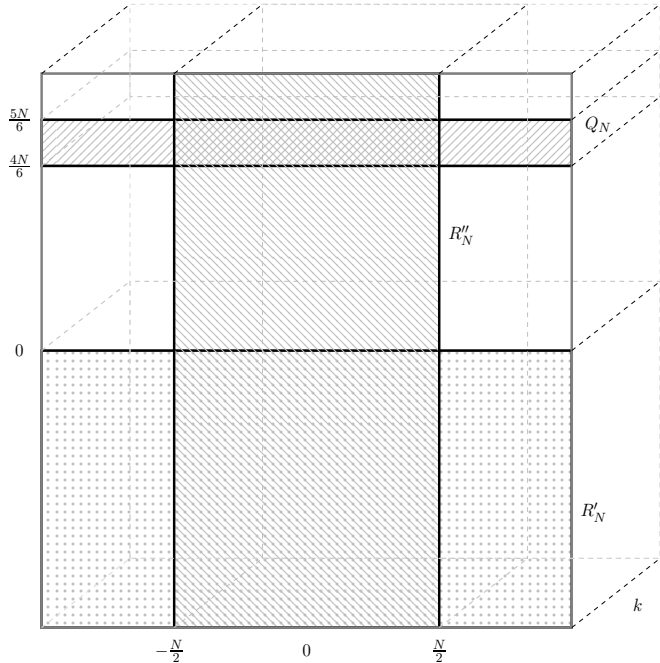
\begin{figure}[ht]
        \centering
        \resizebox{0.55\linewidth}{0.55\linewidth}{
            \begin{tikzpicture}
                \draw[color=black,pattern=dots, pattern color=black!25, ultra thick] (-6,-6) rectangle (6,0);
                \draw[color=black,pattern=north east lines, pattern color=black!25, ultra thick] (-6,4) rectangle (6,5);
                \draw[color=black,pattern=north west lines, pattern color=black!25, ultra thick] (-3,-6) rectangle (3,6);
                \draw[ultra thick, black!50] (-6,-6) rectangle (6,6); 

               \draw[dashed, very thin] (6,6) --(8,7.5);
               \draw[dashed, very thin, black!15] (-6,-6) --(-4,-4.5);
               \draw[dashed, very thin] (6,-6) --(8,-4.5);
               \draw[dashed, very thin, black] (-6,6) --(-4,7.5);

               \draw[dashed, very thin, black] (-3,6) --(-1,7.5);
               \draw[dashed, very thin, black!25] (-3,-6) --(-1,-4.5);
               \draw[dashed, very thin, black!25] (3,-6) --(5,-4.5);
               \draw[dashed, very thin, black] (3,6) --(5,7.5);

               \draw[dashed, very thin, black!25] (-6,4) --(-4,5.5);
               \draw[dashed, very thin, black!25] (-6,5) --(-4,6.5);
               \draw[dashed, very thin, black] (6,4) --(8,5.5);
               \draw[dashed, very thin, black] (6,5) --(8,6.5);
               
               \draw[dashed, very thin, black!25] (-6,0) --(-4,1.5);
               \draw[dashed, very thin, black] (6,0) --(8,1.5); 

                 \draw[color=black!25,dashed, ultra thin] (-4,-4.5) rectangle (8,1.5);
                 \draw[color=black!25,dashed,ultra thin] (-4,5.5) rectangle (8,6.5);
                 \draw[color=black!25,dashed,ultra thin] (-1,-4.5) rectangle (5,7.5);
                 \draw[ultra thin,dashed,black!15] (-4,-4.5) rectangle (8,7.5);
                \node[align=left,color=black] at (6.5,-3.5) {$R_N'$};
                \node[align=left,color=black] at (6.6,4.9) {$Q_N$};
                \node[align=left,color=black] at (3.5,2.5) {$R_N''$};
                \node[align=left] at (0,-6.5) {$0$};
                \node[align=left] at (-6.5,0) {$0$};
                \node[align=left] at (3,-6.5) {$\frac{N}{2}$};
                \node[align=left] at (-3.1,-6.5) {$-\frac{N}{2}$};
                \node[align=left] at (-6.5,4) {$\frac{4N}{6}$};
                \node[align=left] at (-6.5,5) {$\frac{5N}{6}$};
                \node[align=left] at (7.5,-5.5) {$k$};
            \end{tikzpicture}
        }
        \caption{The boxes $\slab{B}_N$, $R'_N$, $R_N''$ and $Q_N$.}
        \label{fig:crossings}\end{figure}    
The following proposition gives a lower bound for the probability that an edge is pivotal for the event  $V(\slab{B}_N)$. In what follows, let $\P_p$ denote the law of Bernoulli percolation with parameter $p$.

\begin{proposition}
    \label{cotaPivotal}
    Let $\tau>0$ and $p>p_c(\S^+_k)$. There are constants $c_4(\tau)>0$ and $c_5(\tau)>0$ such that, for all $N\leq L_\tau(p)$, 
$$\sum_{e \in \mathcal{E}(R_N'')}\P_p(e \mbox{ is piv. to } V(\slab{B}_N))\geq c_4(\tau)\left(\tfrac{p}{2(1-p)}\right)^{5+k}\left(\tfrac{N}{4}\right)^{c_5(\tau)}.$$
\end{proposition}

Proposition \ref{cotaPivotal}, whose proof we delay to Section \ref{CotaPivotalSecao}, together with Russo's formula, allows us to prove Theorem \ref{corr_length_2}.

\begin{proof}[Proof of Theorem \ref{corr_length_2}]
    Fix $p^*>p_c(\S^+_k)$, $N=L_\tau(p^*)$, and write $p(t) = tp^*+(1-t)p_c(\S^+_k)$. By Theorem 3.1 of \cite{tassion} (see also \cite{BS}), there exists a constant $c_6<1$,  such that $\P_{p_c}(V(\slab{B}_N))\geq 1-c_6$, implying
    $$c_6>\P_{p^*}(V(\slab{B}_N))-\P_{p_c}(V(\slab{B}_N)).$$ Integrating and using Russo's formula, we get
    \begin{align*}
        c_6&>\int_0^1 \frac{d}{dt}\P_{p(t)}(V(\slab{B}_N))dt
        =\int_0^1 \frac{d}{dp}\P_{p(t)}(V(\slab{B}_N))p'(t)dt
        \\&=\int_0^1 \sum_{e \in \mathcal{E}(\slab{B}_N))} \P_{p(t)}(e \mbox{ is piv. to }V(\slab{B}_N)) p'(t)dt
        \\&\geq \int_0^1 \sum_{e \in \mathcal{E}(R_N'')} \P_{p(t)}(e \mbox{ is piv. to }V(\slab{B}_N)) p'(t)dt.
\end{align*}
Applying Proposition \ref{cotaPivotal}, we obtain
\begin{align*}
        c_6&> \int_0^1 c_4(\tau)\left(\tfrac{p(t)}{2(1-p(t))}\right)^{5+k}\left(\tfrac{N}{4}\right)^{c_5(\tau)} p'(t)dt
        \\&\geq c_4(\tau)\left(\tfrac{p_c}{2(1-p_c)}\right)^{5+k}\left(\tfrac{N}{4}\right)^{c_5(\tau)} \int_0^1  p'(t)dt
        \\&\geq c_4(\tau)\left(\tfrac{p_c}{2(1-p_c)}\right)^{5+k}\left(\tfrac{N}{4}\right)^{c_5(\tau)} (p^*-p_c)
        \\&=  c_7(\tau)N^{c_5(\tau)}(p^*-p_c).
    \end{align*}
        Since $N=L_\tau(p^*)$, it holds that
$$L_\tau(p^*)\leq c_8(\tau)|p^*-p_c|^{-1/c_5(\tau)},$$
    and the theorem follows with $c_1=c_8$ and $c_2=1/c_5$.
    \end{proof}

\subsection{Proof of Proposition \ref{cotaPivotal}}
\label{CotaPivotalSecao}
Let $R \subseteq \mathcal{V}(\S^+_k)$ and $A,B \subseteq R$ with $A \cap B = \emptyset$. A \first{cutset} separating $A$ and $B$ is a set $F \subseteq \mathcal{E}(R)$ such that every path from $A$ to $B$ must use some edge of $F$. A cutset $F$ is \first{minimal} if, for each non-empty subset $G \subseteq F$, $F \setminus G$ is not a cutset. We say the cutset $F$ is \first{closed} in the configuration $\omega \in \Omega$ if $\omega(e)=0$, for every  $e \in F$. The following three lemmas will help us to prove Proposition \ref{cotaPivotal}.

\begin{lema}\label{cutsetiff} The event
$\{A \overset{R}{\centernot\longleftrightarrow}B \}$ occurs if, and only if, there exists a closed cutset in $R$ separating $A$ and $B$.
\end{lema}
\begin{proof}
    Clearly, if the event $\{A \overset{R}{\longleftrightarrow} B\}$ occurs then, by definition,  there is no closed cutset separating $A$ and $B$ in $R$. On the other hand, if there is no such closed cutset, the set of all closed edges in $\mathcal{E}(R)$ is also not a cutset separating $A$ and $B$, and hence there is an open path joining $A$ and $B$ in $R$.
    \end{proof}

Fix $R = \slab{B}_N$, $A:=[-N,N]\times\{-N\}\times[0,k]$, and $B:=[-N,N]\times\{N\}\times[0,k]$. Let $\mathcal{C}(A):=\{x \in R: x \stackrel{R}{\longleftrightarrow} A\}$ be the open cluster of $A$ in $R$ and $$\frontC A:=\{e =\langle x,y \rangle \in \mathcal{E}(R): x \in \mathcal{C}(A), \omega(e)=0\},$$ the \first{closed boundary} of the cluster of $A$.

\begin{lema}    \label{lowestCutset}
    If the event $\{A \overset{R}{\centernot\longleftrightarrow} B\}$ occurs, then there is a unique minimal closed cutset contained in $\frontC A$.
\end{lema}
\begin{proof}
    By Lemma \ref{cutsetiff}, there is at least one closed cutset, and hence there is at least one minimal cutset. Suppose there are two minimal cutsets $F,G \subseteq \frontC A$. Clearly, $G \not \subset F$ (and $F\not \subset G$) since otherwise, one of them would not be minimal.

    Next, observe that a cutset $F$ splits $R$ in two disjoint subsets, namely
     \begin{equation}
     \label{SideofA}
         \mathcal{L}_A(F):=\{x \in R: x \longleftrightarrow A \mbox{ without using edges of } F\}\end{equation} and 
         \begin{equation}
         \label{SideofB}
             \mathcal{L}_B(F):=\{x \in R: x \longleftrightarrow B \mbox{ without using edges of } F\}.
         \end{equation}

    Let $e =\langle a,b \rangle \in G\setminus F$. Since $G \subseteq \frontC A$, either $a$ or $b$ is connected to $A$ by an open path $\gamma$. Without loss of generality, assume $a$ is such a vertex. By the minimality of $G$, the set $G \setminus \{e\}$ is not a cutset, and therefore, there is a path $\gamma^*$ going from $A$ to $B$ without using edges of $G\setminus\{e\}$ and such that $e\in\mathfrak{E}(\gamma^*)$ (see \eqref{path_edgeset}). Let $\bar{\gamma}^*$ be the portion of $\gamma^*$ connecting $a$ to $B$, and let $\gamma' = \gamma \oplus \bar{\gamma}^*$ be the concatenation of $\gamma$ and $\bar{\gamma}^*$. Note that, since $F$ is a closed cutset and $\gamma$ only has open edges, it holds that $\emptyset \neq \mathfrak{E}(\gamma')\cap F\subseteq \mathfrak{E}(\bar{\gamma}^*)$. This implies that $a,b \in \mathcal{L}_A(F)$. 
    
    Now, take an edge $f = \langle c,d \rangle \in (\mathfrak{E}(\gamma')\cap F)\setminus\{e\}$. Since $f \in \mathfrak{E}(\bar{\gamma}^*)$, it holds that $c,d \in \mathcal{L}_B(G)$. Using the previous argument for the edge $f$, and exchanging the roles of $F$ and $G$, we find that $c,d\in \mathcal{L}_A(G)$. Therefore, $f \in \mathcal{L}_A(G)\cap \mathcal{L}_B(G)$, a contradiction.
\end{proof}
In what follows, when the event $\{A \overset{R}{\centernot\longleftrightarrow} B\}$ occurs, we refer to the minimal closed cutset contained in $\frontC A$ as the \first{nearest cutset to $A$}.


   The next step is to bound the probability of crossing horizontally $n \times 2n$ boxes. We shall see that, as long as $n$ is below the correlation length, such bounds will not depend on $p$. Let $$f_p(n,m) = \P_p(\{0\}\times[0,m]\times[0,k]\overset{[0,n]\times[0,m]\times[0,k]}
    {\looooongleftrightarrow} \{n\}\times[0,m]\times[0,k]).$$
    It follows from the definition in \eqref{cl_def} that, for all $p \in (p_c,1)$, 
    \begin{equation}f_p(2n,n)<1-\tau, \quad \mbox{for all } n<L_\tau(p).
    \label{CorrLenDef}
    \end{equation}

    \begin{lema}\label{exchangingDirections}Let $s>0$. There exists a function $u(\tau)>0$ such that, for all $p \in (p_c,1-s]$,
	$$f_p(n,2n)<1-u(\tau), \quad \mbox{for all } n < L_\tau(p).$$
    \end{lema}

    \begin{proof} The following is based on the proof of Theorem 3.19 of \cite{tassion}.
	For $p \in (p_c,1-s]$, define the functions $\mathcal{M}_p = \mathcal{M}_p(\tau)$ and $\mathcal{N}_p = \mathcal{N}_p(\tau)$ as
	\begin{align*}
	    \mathcal{M}_p:=\max_{n<L_\tau(p)}\left\{f_p(n,2n)\right\}\quad \mbox{ and }\quad 
            \mathcal{N}_p:=\max_{n<L_\tau(p)}\left\{f_p(2n,n)\right\}.
	\end{align*}
 We claim that $\sup_{p_c<p<1-s} \mathcal{M}_p<1$, which implies the existence of a function $u(\tau)>0$ such that, for all $p\in(p_c,1-s]$,
	$$f_p(n,2n)\leq 1-u(\tau), \quad \mbox{for all } n<L_\tau(p).$$
    The proof proceeds by contradiction, showing that, if $\sup_{p_c<p\leq1-s}\mathcal{M}_p=1$, then $\sup_{p_c<p\leq1-s}\mathcal{N}_p=1$, which is not possible according to \eqref{CorrLenDef}.

	Assume therefore that $\sup_{p_c<p \leq 1-s}\mathcal{M}_p=1$. By Theorem 3.11 of \cite{tassion}, there exists a continuous increasing function $h:[0,1]\mapsto [0,1]$ with $h(1)=1$ such that
    $$f_p(3n,4n)\geq h(f_p(n,2n)),$$ for all $p_c<p\leq1-s$. Hence, 
    
	\begin{equation}
		\sup_{p_c<p\leq1-s}\max_{n<L_\tau(p)} f_p(3n,4n)=1.
		\label{sup34}
	\end{equation}
	Observe  that $L_\tau(p)<\infty$ for all $p>p_c$. Also, since $L_\tau(p)$ is decreasing in $p$, we have $L_\tau(p)\leq L_\tau(p^*)$ whenever $p_c<p^*<p<p'$, yielding
    \begin{equation}
        \label{corLengIneq}
        \max_{n<L_\tau(p)}f_p(3n,4n)\leq \max_{n<L_\tau(p^*)}f_{p'}(3n,4n).
    \end{equation}

    Therefore, if \eqref{sup34} is true, there must exist a sequence $p_t\downarrow p_c$, and a sequence $n_t<L_\tau(p_t)$ with $n_t \to \infty$, such that 
    \begin{equation}
        \lim_{t\to \infty}\left[f_{p_t}(3n_t,4n_t)\right]=1.
        \label{sup34limit}
    \end{equation}
    
    As in Theorem 3.19 of \cite{tassion}, for all $\zeta>0$, there is a constant $c_9 = c_9(\zeta)>0$ such that, if $f_{p}(3c_9n,4c_9n)$ is close enough to $1$, then $f_{p}(6c_9n,5c_9n)>1-\zeta$.    By $\eqref{corLengIneq}$ and \eqref{sup34limit}, there are subsequences $p_{t_i}$ and $n_{t_i}$ with $p_c<p_{t_i}\leq p_i$ and $n_{t_i}\leq \frac{L_\tau(p_{t_i})}{c_9}$, such that $P_{p_{t_i}}(3c_9n_{t_i},4c_9n_{t_i}) \rightarrow 1$ as $t_i$ goes to infinity. Hence, for $t_i$ large enough, we obtain
        $$f_{p_{t_i}}(6c_9n_{t_i},5c_9n_{t_i})\geq 1-\zeta.$$

    Finally, a second application of Theorem 3.11 of \cite{tassion} gives that $f_{p_i}(2c_9n_{t_i},c_9n_{t_i})>1-\tau$ for all $t_i$ large enough. But this is a contradiction, since $c_9n_{t_i}\leq L_\tau(p_{t_i})$ for all $t_i$. This completes the proof.
    \end{proof}

\begin{corollary}\label{CotaUmBraco} Let $\tau>0$. There exists a constant $c_5(\tau)>0$ such that 
\begin{equation*}
\P_p(0 \overset{\slab{B}_N}{\longleftrightarrow }\frontA \slab{B}_N)\leq N^{-c_5(\tau)},\; \mbox{ for }N<L_\tau(p).
\end{equation*}
\end{corollary}
This follows from Lemma \ref{exchangingDirections} and classical annuli crossing techniques for planar percolation; see, e.g., the proof of Theorem 6 in \cite{bollobas}.

We are now in a position to prove Proposition \ref{cotaPivotal}.

\begin{proof}[Proof of Proposition \ref{cotaPivotal}]

Fix $p^*>p_c(\S^+_k)$. Let $p_c(\S_k^+)<p<p^*$, and $N<L_\tau(p)$. Define the sets $BO=[-N,N]\times \{-N\}\times [0,k]$, $TO=[-N,N]\times\{N\}\times[0,k]$, $LS = \{-N\}\times[-N,N]\times[0,k]$, and $RS = \{N\}\times[-N,N]\times [0,k]$, the bottom, top, left side and right side of the box $\slab{B}_N$, respectively.


    Recall the definition of the sets $R_N'$ and $Q_N$ at the beginning of this section, and define the event $E_N:=H(Q_N)\cap V(R_N')^c$ (see \eqref{Hor_cross} and \eqref{ver_cross}). If the event $V(R_N')^c$ occurs, let $\Phi \subseteq \mathcal{E}(R'_N)$ be the nearest cutset to $BO$ in $R_N'$, separating the top and bottom of $R_N'$ (note that such nearest cutset to $BO$ exists by Lemmas \ref{cutsetiff} and \ref{lowestCutset}), and let $$\frontD\Phi:=\{x \in \mathcal{L}_{TO}(\Phi): \langle x,y \rangle \in \Phi \mbox{ for some } y \in \slab{B}_N\},$$ where $\mathcal{L}_{TO}(\Phi)$ is defined as in \eqref{SideofA} and \eqref{SideofB}. Writing $F_N= E_N \cap \left\{
        TO(\slab{B}_N)\overset{R_N''}{\longleftrightarrow} \frontD\Phi\right\}$ and conditioning on $\Phi$, we obtain
        \begin{align*}
            \P_p(F_N)&=\sum_{S\subseteq \E(R_N')}\P_p(F_N|\Phi=S)\P_p(\Phi=S)
            \\&=\sum_{S\subseteq \E(R_N')}\P_p\left(E_N\cap \left\{TO(\slab{B}_N)\overset{R_N''}{\longleftrightarrow} \frontD\Phi\right\}\Big|\Phi=S\right)\P_p(\Phi=S)
            \\&=\sum_{S\subseteq \E(R_N')}\P_p\left(H(Q_N)\cap\left\{TO(\slab{B}_N)\overset{R_N''}{\longleftrightarrow} \frontD\Phi\right\}\Big|\Phi=S\right)\P_p(\Phi=S).
        \end{align*}
        
        Observe that, conditioned on $\{\Phi = S\}$, the events $H(Q_N)$ and $\{TO(\slab{B}_N)\overset{R_N''}{\longleftrightarrow}\partial^{out}\Phi\}$ are measurable with respect to $\{\omega(e)\}_{e \in \mathcal{E}(\mathcal{L}_{TO}(S))}$, where $\mathcal{E}(A)$ denotes the set of edges with both endpoints in $A\subset\mathbb{S}_k ^+$. Thus,  
        an application of the FKG inequality yields        
        \begin{align*}
            \P_p(F_N)&\geq \sum_{S\subseteq E(R_N')}\P_p\bigg(H(Q_N)\Big|\Phi=S\bigg)\P_p\left(TO(\slab{B}_N)\overset{R_N''}{\longleftrightarrow} \frontD\Phi\big|\Phi=S\right)\P_p\bigg(\Phi=S\bigg)
            \\&\geq \sum_{S\subseteq \E(R_N')}\P_p(H(Q_N)|\Phi=S)\P_p(V(R_N''))\P_p(\Phi=S)
            \\&= \P_p(H(Q_N))\P_p(V(R_N''))\sum_{S\subseteq E(R_N')}\P_p(\Phi=S)
            \\&= \P_p(H(Q_N))\P_p(V(R_N''))\P_p(V(R_N')^c)
            \\&\geq c_4(\tau),
        \end{align*}    
for some constant $c_4(\tau)>0$. The first equality follows by independence between the events $H(Q_N)$ and $\{\Phi=S\}$, while the last inequality is a consequence of Theorem 3.1 of \cite{tassion} and Lemma \ref{exchangingDirections}.

In the following, $A \circ B$ denotes the disjoint occurrence of the events $A$ and $B$ (see definition on page 160 of \cite{Chayes}). Fix a configuration $\omega \in F_N$. We make a series of local changes in $\omega$ in order to find a new configuration $\omega'$ and an edge $\langle v,z \rangle\in \mathcal{E}(R_N'')$ such that 
\begin{enumerate}[label = (\roman*)]
     \item $\omega' \in \{v \overset{\slab{B}_N}{\longleftrightarrow}TO(\slab{B}_N) \}\circ \{v \overset{\slab{B}_N}{\longleftrightarrow} LS(\slab{B}_N)\}$,
    \item $\omega' \in \{z \overset{\slab{B}_N}{\longleftrightarrow} BO(\slab{B}_N)\}$,
    \item $|\{e \in \E(\slab{B}_N):\omega(e) \neq \omega'(e)\}|\leq  k+5$,
    \item there is a closed cutset in $\omega'$ separating $BO(\slab{B}_N)$ from $TO(\slab{B}_N)$.
\end{enumerate}

Observe that items (i), (ii), and (iv) ensure that the edge $\langle v,z\rangle$ is pivotal for $V(\slab{B}_N)$ in the configuration $\omega'$. We perform the necessary local changes in two steps.

\textbf{Step 1: } Since $\omega \in F_N$, there is an edge $\langle x_0, y_0 \rangle  \in \Phi$, and open paths $\gamma_a$, $\gamma_b$, and $\gamma_c$, such that the following holds (see Figure \ref{fig:CrossingsProof}):
       \begin{itemize}
           \item $\gamma_a \subset R_n''$ joins $x_0$ to $TO$,
           \item $\gamma_b$ joins $y_0$ to $BO$,
           \item $\gamma_c$ crosses the box $Q_N$.
       \end{itemize}

    \begin{figure}[ht]
        \centering
        \resizebox{0.6\linewidth}{0.6\linewidth}{
            \begin{tikzpicture}

                \draw[color=black!5, ultra thick] (-6,-6) rectangle (6,0);
                \draw[color=black!5, ultra thick] (-6,4) rectangle (6,5);
                \draw[color=black!5, ultra thick] (-3,-6) rectangle (3,6);
                \draw[ultra thick, black!15] (-6,-6) rectangle (6,6); 

                \draw[ very thin,black!15] (6,6) --(8,7.5);
                \draw[ very thin, black!15] (-6,-6) --(-4,-4.5);
                \draw[ very thin,black!15] (6,-6) --(8,-4.5);
                \draw[ very thin, black!15] (-6,6) --(-4,7.5);

                \draw[dashed, very thin, black!5] (-3,6) --(-1,7.5);
                \draw[dashed, very thin, black!5] (-3,-6) --(-1,-4.5);
                \draw[dashed, very thin, black!5] (3,-6) --(5,-4.5);
                \draw[dashed, very thin, black!5] (3,6) --(5,7.5);

                \draw[dashed, very thin, black!5] (-6,4) --(-4,5.5);
                \draw[dashed, very thin, black!5] (-6,5) --(-4,6.5);
                \draw[dashed, very thin, black!5] (6,4) --(8,5.5);
                \draw[dashed, very thin, black!5] (6,5) --(8,6.5);
               
                \draw[dashed, very thin, black!5] (-6,0) --(-4,1.5);
                \draw[dashed, very thin, black!5] (6,0) --(8,1.5); 

                \draw[color=black!5,dashed, ultra thin] (-4,-4.5) rectangle (8,1.5);
                \draw[color=black!5,dashed,ultra thin] (-4,5.5) rectangle (8,6.5);
                \draw[color=black!5,dashed,ultra thin] (-1,-4.5) rectangle (5,7.5);
                \draw[ultra thin,black!15] (-4,-4.5) rectangle (8,7.5);               

                \draw [color=black!35,pattern=north west lines, pattern color=black!25, ultra thick] (-6,-1.5)--(-5,-0.75)--(7,-0.75)--(6,-1.5) --cycle;

                \draw [color=black!35,pattern=north west lines, pattern color=black!25, ultra thick] (-5,-0.75)--(-5,-3.0)--(7,-3.0)--(7,-0.75) --cycle;

                \draw [color=black!35,pattern=north west lines, pattern color=black!25, ultra thick] (-5,-3.0)--(-4,-2.25)--(8,-2.25)--(7,-3.0) --cycle;

                 \draw[color=black, very thick,loosely dotted](-2,6.75)--(-2,3.75)--(0,3.75)--(1,4.5)--(1,5.5)--(-1,5.5)--(-3,4)--(-3,2)--(-1,3.5)--(1,3.5)--(0,2.75)--(0,0.75)--(1,1.5)--(1,-2.5);

                 \draw[color=black, ultra thick] (1,-2.5)--(1,-3.5);
                
                 \draw[color=black, very thick,loosely dashed](1,-3.5)--(-1,-5)--(0,-5)--(0,-2)--(-5,-2)--(-5,-4)--(-4,-3.25)--(-4,-5.25);
                 
                 \draw[color=black, very thick](-5,5.75)--(3,5.75)--(2,5)--(2,4)--(4,5.5)--(5,5.5)--(5,6.5)--(7,6.5)--(7,5.5)--(6,4.75)--(7,4.75);

                 \node[align=left,color=black] at (8,-1.5){\Huge$\Phi$};
                 \node[align=left,color=black] at (1.5,2.5){\Huge $\gamma_a$};
                 \node[align=left,color=black] at (-0.5,-5.5){\Huge $\gamma_b$};
                 \node[align=left,color=black] at (8.7325,5.825){\Huge $\gamma_c$};

                 \filldraw [gray] (1,-2.5) circle (3pt);
                 \filldraw [gray] (1,-3.5) circle (3pt);
                 \node[align=left] at (1.5,-2.5){\LARGE $x_0$};
                 \node[align=left] at (1.5,-3.5){\LARGE $y_0$};
                 
            \end{tikzpicture}
        }
           \caption{A sketch of the paths $\gamma_a, \gamma_b$ and $\gamma_c$  on the event $F_N$.}
           \label{fig:CrossingsProof}
       \end{figure}
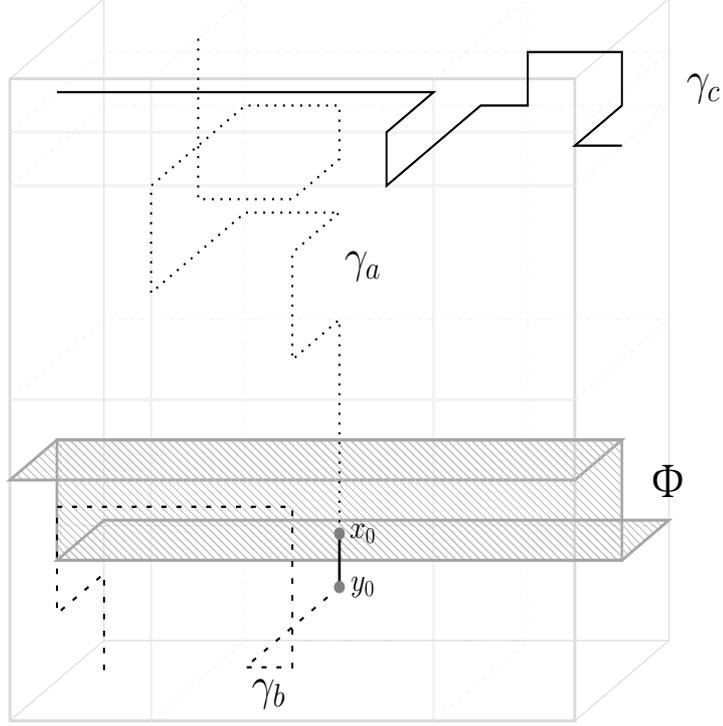

    Let $\pi:\mathbb{Z}^2\times[0,k]\to \mathbb{Z}^2$, $\pi(x,y,z)=(x,y)$, be the projection of the slab onto $\mathbb{Z}^2$. If $\gamma$ is a path, let $\pi(\gamma)$ be the projection of the vertices of $\gamma$ onto $\Z^2$. Fix such $\gamma_a$, $\gamma_b$, and $\gamma_c$, and observe that the path $\gamma^* = \gamma_a\oplus \langle x_0,y_0 \rangle\oplus \gamma_b$ crosses the box $\slab{B}_N$ vertically, and $\gamma_c$ crosses it horizontally. Moreover, $\pi(\gamma_a)\cap \pi(\gamma_c)\neq \emptyset$.

        Fix a vertex $w$ in $\pi(\gamma_a)\cap \pi(\gamma_c)$; if there is a choice for $w$ we pick it according to some arbitrary ordering of $\mathbb{Z}^2$. Consider a new configuration $\bar{\omega}$ as 
        $$\bar{\omega}(\langle a,b \rangle) = \left\{\begin{array}{ll}
            1 &\mbox{ if } \pi(a) = \pi(b)=w,\\
            \omega(\langle a,b \rangle) &\mbox{otherwise}.
        \end{array}\right.$$
 In this new configuration, the paths $\gamma_a$ and $\gamma_c$ are connected, thereby joining $x_0$ to $LS$, $RS$, and $TO$. Observe that the vertex $y_0$ is still connected to $BO$ in $\bar{\omega}$, and  that $|\{e \in \mathcal{E}(\slab{B}_N): \omega(e)\neq \bar{\omega}(e)\}|\leq k$.


        \textbf{Step 2:} Define the following collection of paths: 
    \begin{align*}
        \Gamma_T&:=\{\gamma \mbox{ is an open path in }\slab{B}_N: \gamma \mbox{ connects } x_0 \mbox{ to }TO\},
        \\\Gamma_S&:=\{\gamma \mbox{ is an open path in }\slab{B}_N: \gamma \mbox{ connects } x_0 \mbox{ to }LS\},
    \end{align*}
    and, for $\bar{\omega}$ as in  Step 1, let $$I:=\{u \in \slab{B}_N: u \in \gamma_t \cap \gamma_s, \mbox{ for all } \gamma_t \in \Gamma_T, \gamma_s \in \Gamma_S\}.$$

    Clearly, $x_0 \in I(\bar{\omega})$. If $I(\bar{\omega}) = \{x_0\}$, then there is at least one path $\gamma_s \in \Gamma_S$ and $\gamma_t \in \Gamma_T$ with $\gamma_s \cap \gamma_t = \{x_0\}$. Hence, we can make $\langle v,z \rangle = \langle x_0,y_0 \rangle$ and $\omega'=\bar{\omega}$.

    Assume $|I(\bar{\omega})|\geq 2$. Given a path $\gamma \in \Gamma_T \cup \Gamma_S$ and $u,u'\in I(\bar{\omega})$ (in particular $u,u' \in \gamma$), we write $u\leq_\gamma u'$ when $u$ is closer to $x_0$ than $u'$ in $\gamma$, and we define $u \leq_Iu'$ when $u \leq_\gamma u'$ for some $\gamma\in \Gamma_T\cup \Gamma_S$. We will show later that the definition of $\leq_I$ does not depend on $\gamma$; for now, assume that this is the case.

    Let $v_0$ be the largest vertex of $I(\bar{\omega})$ according to the ordering $\leq_I$. Since  $\gamma_a\in\Gamma_T(\bar{\omega})$,  $\gamma_a\subset R_N''$, and $v_0\in I(\bar{\omega})\subset \gamma_a$, we have $v_0\in R_N''$. As $v_0$ is the last element of $I(\bar{\omega})$, there are three paths $\gamma_s$, $\gamma_t$, $\gamma_x$, such that $\gamma_s$ joins $v_0$ to $LS$, $\gamma_t$ joins $v_0$ to $TO$, $\gamma_x$ joins $v_0$ to $x_0$, and
    $$\gamma_s\cap I = \gamma_t\cap I=  \gamma_s \cap \gamma_t\cap \gamma_x = \{v_0\}.$$

    The desired configuration $\omega'$ is obtained as follows: let $\langle z,v \rangle$ be the last edge of the path $\gamma_x$, with $v=v_0$. Close all edges adjacent to $v_0$, except those in $\mathfrak{E}(\gamma_s)\cup \mathfrak{E}(\gamma_t)$ (at most 4 edges), and open the edge $\langle x_0,y_0 \rangle$. Note that $\langle z,v \rangle$ is closed in $\omega'$.

    Observe that the paths $\gamma_s$ and $\gamma_t$ ensure that $v$ satisfies condition (i). Because $\langle x_0,y_0 \rangle$ is open and $z$ is connected to $x_0$ through $\gamma_x$, condition (ii) is also valid. Since $|\{e \in \mathcal{E}(\slab{B}_N): \omega'(e)\neq \bar{\omega}(e)\}|\leq 5$, we have that $\omega'$ differs from $\omega$ in at most $k+5$ edges, so (iii) is also true.

    Finally, since $\bar{\omega} \in V(R_N')^c$ and $\langle x_0,y_0 \rangle$ is the only edge in $\mathcal{E}(R_N')$ that we opened to get $\omega'$, every possible vertical crossing of $\slab{B}_N$ in $\omega'$ must use the edge $\langle x_0,y_0 \rangle$. However, every path connecting $x_0$ to $TO$ must pass through every vertex in the set $I(\omega')$, in particular the vertex $v$, and since we closed all edges adjacent to $v$ that could join $x_0$ to $v$ in $\omega'$, no such open path exists. Therefore, $\omega' \notin V(\slab{B}_N)$, which implies that (iv) is also true.

    Since $v \in R_N''$, the path $\gamma_s$ has length at least $\frac{N}{4}$. Also, by construction, 
    $$\omega' \in \{\langle v,z \rangle \mbox{ is pivotal for }V(\slab{B}_N)\}\circ\{v \overset{\slab{B}_N}{\longleftrightarrow} LS\}.$$
    By the BKR inequality and Corollary \ref{CotaUmBraco}, there exists a constant $c_5(\tau)>0$ such that
    \begin{align*}
        \P_p(\{\langle v,z \rangle\mbox{ piv. to } V(\slab{B}_N)\}\circ \{v \overset{\slab{B}_N}{\longleftrightarrow} LS\})
        &\leq \P_p(\langle v,z \rangle\mbox{ piv. to } V(\slab{B}_N))\P_p(\{v \overset{\slab{B}_N}{\longleftrightarrow} LS\})
        \\&\leq \P_p(\langle v,z \rangle\mbox{ piv. to } V(\slab{B}_N))\P_p(\{v \overset{\slab{B}_{N/4}}{\longleftrightarrow}\frontA{B}_\frac{N}{4} (v)\})\\
        &\leq \P_p(\langle v,z \rangle\mbox{ piv. to } V(\slab{B}_N))\left(\tfrac{N}{4}\right)^{-c_5(\tau)}.
    \end{align*}

For an edge $e \in \E(\slab{B}_N)$ pivotal for $V(\slab{B}_N)$, let $t(e) \in \slab{B}_N$ be the endpoint of $e$ joined to $TO$. Writing $F'_N = \{\omega^*: \exists\, \omega \in F_N\mbox{ with }\omega'(\omega)=\omega^*\}$, we find
\begin{align*}
       \sum_{e\in \mathcal{E}(R_N'')}\P_p(\{e &\mbox{ piv. to }V(\slab{B}_N)\})\geq \sum_{e \in \E(R''_N)}\P_p(\{ e \mbox{ piv. to }V(\slab{B}_N)\}\circ \{t(e) \overset{\slab{B}_N}{\longleftrightarrow} LS\})\left(\tfrac{N}{4}\right)^{c_5(\tau)}
        \\&\geq \P_p\left(\bigcup_{e \in \E(R_N'')}\{ e \mbox{ piv. to }V(\slab{B}_N)\}\circ \{t(e) \overset{\slab{B}_N}{\longleftrightarrow} LS\}\right)\left(\tfrac{N}{4}\right)^{c_5(\tau)}
       \\&\geq \sum_{\omega^*\in \,F'_N}\P_p(\omega^*)\left(\tfrac{N}{4}\right)^{c_5(\tau)}
        = \sum_{\omega \in F_N}\P_p(\omega'(\omega))\tfrac{1}{|(\omega')^{-1}(\omega)|}\left(\tfrac{N}{4}\right)^{c_5(\tau)}
        \\&\geq \sum_{\omega \in F_N}\P_p(\omega)\left(\tfrac{p}{1-p}\right)^{5+k}\tfrac{1}{2^{5+k}}\left(\tfrac{N}{4}\right)^{c_5(\tau)}
        =\P_p(F_N)\left(\tfrac{p}{2(1-p)}\right)^{5+k}\left(\tfrac{N}{4}\right)^{c_5(\tau)}
        \\&\geq c_4(\tau)\left(\tfrac{p}{2(1-p)}\right)^{5+k}\left(\tfrac{N}{4}\right)^{c_5(\tau)},
    \end{align*}
where the first equality follows since there are $|(\omega')^{-1}(\omega)|$ configurations that are transformed into the same $\omega'$, and the fourth inequality comes from the fact that the transformation $\omega'$ changes at most $k+5$ edges.

To conclude the proof, it remains to show that $\leq_{I}$ does not depend on the choice of $\gamma\in \Gamma_T\cup \Gamma_S$. Suppose there are two paths $\gamma, \gamma' \in \Gamma_T \cup \Gamma_S$ and $u, u' \in I$ such that $u\leq_\gamma u'$ and $u' \leq_{\gamma'}u$. Then the path $\gamma$ and $\gamma'$ can be split into three paths $\gamma = \gamma_{x,u}\oplus \gamma_{u,u'}\oplus \gamma_{u',z}$ and $\gamma' = \gamma'_{x,u'}\oplus \gamma'_{u',u}\oplus \gamma'_{u,w}$, where $\gamma_{a,b}$ denotes a path joining $a$ to $b$. But if that is the case, then there exists a path $\gamma'' = \gamma_{x,u}\oplus \gamma'_{u,w}\in \Gamma_T\cup \Gamma_S$ that does not use the vertex $u'$, which is absurd. Therefore, the ordering $\leq_{I}$ does not depend on the choice of $\gamma$.
\end{proof}

\section{Proof of Proposition \ref{theoremEdge}}
\label{boundsEdge}

In this section, we prove Proposition \ref{theoremEdge}. Our proof is built on a series of results drawn from Russo-Seymour-Welsh theory (see \cite{Russo1978} and \cite{SEYMOUR1978227}) and also from \cite{tassion}. 

Let $A$, $B$, $S \subset \S_k^+$. If $\omega \in \{A \stackrel{S}{\longleftrightarrow}B\}$, we denote by $\gamma_{A,B}(\omega)$ the smallest open path joining $A$ to $B$ in $S$, according to Definition \ref{pathOrdering}. If no such connection exists, we set $\gamma_{A,B}(\omega)=\emptyset$.

We begin by stating a "gluing lemma", whose proof is omitted here, as it closely follows the argument provided for Theorem 3.8 in \cite{tassion}.

\begin{lema}[Gluing Lemma]\label{GluingLemmaGenerico}
Fix $s>0$.
Let $S,R \subseteq \S^+_k$ be connected non-empty sets such that $S\cap R$ is a box $[x_1,x_2]\times[y_1,y_2]\times[0,k]$ with $x_2-x_1\geq 5$ and $y_2-y_1\geq 5$. Let $A,B \subset S$ and $C,D \subset R$ be such that the projections of $A, B, C, D$ on $\mathbb{Z}_+^2$ are disjoint, and such that the projection (on $\mathbb{Z}_+^2$) of each path in $S$ from $A$ to $B$  intersects the projection of each path in $R$ from $C$ to $D$. Then there exists a continuous increasing function $h:[0,1]\to [0,1]$ with $h(0)=0$, $h(1)=1$, such that for all $(p,q) \in [s,1-s]^2$
\begin{equation}
\label{Eq:gluingLemma}
    \P^\Lambda_{p,q}(\gamma_{A,B}\stackrel{S\cup R}{\longleftrightarrow}C)\geq h(\P^\Lambda_{p,q}(A\stackrel{S}{\longleftrightarrow}B)\wedge \P^\Lambda_{p,q}(C\stackrel{R}{\longleftrightarrow}D)).
\end{equation}
\end{lema}
 \begin{remark}\label{remark3} As in Theorem 3.10 of \cite{tassion}, Lemma \ref{GluingLemmaGenerico} also holds when $C$ is the random set $\Gamma_x^n(\omega)$; see  \eqref{DefinitionGamma}.
 \end{remark}

Let $f \in \E(\Z_+^2)$. Recall the definition of the event $A_n(f)$ given in \eqref{DefinitionA}, which indicates that a favored renormalized edge is open. With the aid of Lemma \ref{GluingLemmaGenerico}, we will show the following:

\begin{lema}\label{gluingLemmaEdge}  Fix $s>0$. For all $(p,q) \in [s,1-s]^2$ and for all $n\geq 1$, it holds that
$$\P^{\Lambda}_{p, q}(A_n(f))\geq h_3(\P^\Lambda_{p,q}(D_n)\wedge \P^\Lambda_{p,q}(V_n)\wedge \P^\Lambda_{p,q}(H_n)),$$
where $h_3(x) = h(\min\{x,h(x)\})$, with $h$ given by Lemma \ref{GluingLemmaGenerico}.
\end{lema}

\begin{proof}
The proof of Lemma \ref{gluingLemmaEdge} consists of two applications of Lemma \ref{GluingLemmaGenerico}. Assume without loss of generality that $f = \langle x, x+(1,0) \rangle$ (the vertical case is analogous, replacing $H_n(x)$ with $V_n(x)$). 

Following the notation in Lemma \ref{GluingLemmaGenerico}, let $S=\slab{B}_n(x)$, $R=\slab{HR}_n(x)$,  $A$ = $\slab{LB}_n(x)$, $B = \slab{RB}_n(x)$, $C = \slab{LS}_n(x)$ and $D = \slab{RS}_n(x+(1,0))$.
Lemma \ref{GluingLemmaGenerico} gives a function $h$ such that
\begin{equation}
    \label{firstGluing}
\P_{p,q}^\Lambda(\Gamma^n_x(\omega)\stackrel{S \cup R}{\longleftrightarrow} D)\geq h(\P_{p,q}^\Lambda(D_n(x))\wedge \P_{p,q}^\Lambda(H_n(x))),
\end{equation}
where $\Gamma^n_x(\omega)$ is the smallest path joining $A$ to $B$ in $S$, as defined in Section \ref{SectionOverview}.

\begin{figure}[ht]
        \centering
        \resizebox{0.7\linewidth}{!}{
            \begin{tikzpicture}

                \draw[line width = 2mm,color=black] (18,-3) rectangle (18,3);
                \draw[color=black,line width = 2mm] (-6,-3)--(-6,3);

                \draw[color=black, ultra thick,loosely dashed] (-6,-6) rectangle (-3,-3);
                \draw[color=black, ultra thick,loosely dashed] (3,3) rectangle (6,6);

                \draw[color=black!75, ultra thick] (-6+12,-6) rectangle (6+12,6);
                \draw[color=black, ultra thick,line width = 1mm] (-6+12,-6) rectangle (-3+12,-3);
                \draw[color=black, ultra thick,line width = 1mm] (3+12,3) rectangle (6+12,6);

                 \draw[color=black,pattern = north west lines, pattern color = black!35] (-6,-3) rectangle (18,3);

                \draw[color=black!50,pattern = north east lines,pattern color = black!25] (-6,-6) -- (6,-6)--(6,-3)--(18,-3)--(18,3)--(6,3)--(6,6)--(-6,6)--(-6,-6);

                
                \node[align = left,color=black!75,scale=5] at (-4.5+12,-4.5) {$A'$};
                \node[align = left,color=black!75,scale=5] at (4.5+12,4.5) {$B'$};

                \node[align = left,color=black!75,scale=5] at (-4.5,-4.5) {$A$};
                \node[align = left,color=black!75,scale=5] at (4.5,4.5) {$B$};
                \node[align=left,color=black!75,scale=5] at (-7.25,0) {$C$};

                 \node[align=left,color=black!75,scale=5] at (21,0) {$D=D'$};
                
                \node[align = left,color=black!75,scale=4] at (-1.25,2){$C' = \Gamma^n_x(\omega)$};
            
                \draw [color=black,loosely dashed] (-2 ,0) to [ curve through ={(3,0)..(7,2)..(10,1)..(12,1)..(14,-2)}] (18,-1);

                \draw [very thick,color=black,line width = 1mm] (-3,-4) to [ curve through ={(-2,-3)..(-4,-2)..(-2,0)..(0,1)..(2,1)}] (4.5,3);

                \draw [very thick,color=black] (-3+12,-4) to [ curve through ={(-2+12,-5)..(-1+12,-2)..(1+12,-1)..(2+12,-2)..(3+12,1)}] (4.5+12,3);

            \end{tikzpicture}
        }
        \caption{The sets $A$, $B$, $C$, $A'$, $B'$, $C'$, and $D'$. In dashed line a path making the event $\{\Gamma_n^{x}(\omega) \overset{S \cup R}{\longleftrightarrow} D\}$. In a continuous line, a path making the event $D_n(x+(1,0))$. The dark grey region is the set $R=HR_n(x)$. The union of the shaded regions is the set $R' = S \cup R$.}
        \label{fig:AplicacaoGL1}
    \end{figure} 

    \FloatBarrier

Also, let $S' = \slab{B}_n(x+(1,0))$ and $R' = \slab{B}_n(x)\cup \slab{HR}_n(x)$. Let $A'$ = $\slab{LB}_n(x+(1,0))$, $B' = \slab{RB}_n(x+(1,0))$, $C' = \Gamma^n_x(\omega)$, and $D' = \slab{RS}_n(x+(1,0))$; see Figure \ref{fig:AplicacaoGL1}. A second application of Lemma \ref{GluingLemmaGenerico} (see also Remark \ref{remark3}) yields
\begin{equation}
    \label{secondGluing}
    \P_{p,q}^\Lambda(\Gamma^n_{x+(1,0)}(\omega)\stackrel{S' \cup R'}{\longleftrightarrow} \Gamma^n_x(\omega))\geq h\left(\P_{p,q}^\Lambda(D_n(x+(1,0))\wedge \P_{p,q}^\Lambda(\Gamma^n_x(\omega)\stackrel{S \cup R}{\longleftrightarrow} D)\right).
\end{equation}

Observe that the left-hand side of \eqref{secondGluing} is exactly $\P_{p,q}^\Lambda(A_n(f))$. Combining it with \eqref{firstGluing}, we find
\begin{align*}    
\P_{p,q}^\Lambda(A_n(f))&\geq h\left[(\P_{p,q}^\Lambda(D_n(x+(1,0))\wedge h\left(\P_{p,q}^\Lambda(D_n(x))\wedge \P_{p,q}^\Lambda(H_n(x))\right)\right]
\\&\geq h\left(\min\left\{\P_{p,q}^\Lambda(D_n)\wedge \P_{p,q}^\Lambda(H_n),h(\P_{p,q}^\Lambda(D_n)\wedge \P_{p,q}^\Lambda(H_n))\right\}\right)
\\&=h_3(\P_{p,q}^\Lambda(D_n)\wedge \P_{p,q}^\Lambda(H_n)),
\end{align*}
since the events $D_n$ and $H_n$ are translation invariant. The calculations are analogous for a vertical edge, yielding the result. 
\end{proof}

Next, we will apply Lemma \ref{GluingLemmaGenerico} to compare the probabilities of the crossing event $H([0,2n]\times[0,n]\times[0,k])$ under the anisotropic measure $\P_{p,q}^\Lambda$ and the isotropic measure $\P_{p_n}$, for some suitable choice of $p_n \in (p,q)$.

\begin{lema}\label{aizenmanGrimmettStep} 
Fix $n\in\mathbb{N}$ and $\lambda<1$. Let $\frac{n^{\lambda}}{2}<t\leq n$, and consider $R_t = [0,2t]\times[0,t]\times[0,k]$. Assume that $[0,n]$ and $[n,2n]$ are both $\lambda$-good intervals. For all $\varepsilon>0$, there exist $\delta = \delta(\varepsilon, n,\lambda)$, $c_{14} = c_{14}(\varepsilon)>0$, and $c_{15}>0$, such that  
    \begin{equation*}\label{eq:DominacaoLema7}
    \P^\Lambda_{p_c-\delta,p_c+\varepsilon}(H(R_t))\geq \P_{p_n}(H(R_t)),\end{equation*}
    where $p_n = p_c+c_{14}n^{-\lambda c_{15}}$.
\end{lema}

\begin{proof}
We follow the arguments of \cite{Brochette}, making the necessary extensions for the slab $\S_k^+$. Fix an environment $\Lambda$, as defined in Section \ref{ModelandMainResult}, and write $$Enh(\Lambda) = \E(\Lambda \times \Z_+\times\{0\})$$ for the set of \first{enhanced} edges. Given two edges $e = \langle e_1,e_2 \rangle$ and $f=\langle f_1,f_2 \rangle$ in $\E(\S_k^+)$, we define the distance between them by
$\max_i\{\min_j \|e_i-f_j\|\}$. Fix an edge $e \in \E(R_t)\setminus Enh(\Lambda)$, whose probability of being open under $\P_{p,q}^\Lambda$ is $p$, and let $f \in \mathcal{E}(R_t)\cap Enh(\Lambda)$ be the closest edge to $e$. If there is a choice for $f$, pick it according to some arbitrary order. Since $[0,n]$ and $[n, 2n]$ are $\lambda$-good, the distance between $f$ and $e$ is smaller than $n^\lambda+k$. 

Write $\slab{B}_d$ for a box such that $e \in \E(\slab{B}_d)$, $f \in \E(\frontA\slab{B}_d)$, and $d<n^\lambda$. For simplicity, we assume that $d>200$. Let $B'_d = B_{\frac{d}{2}}$ and $B''_d = B_{d-100}$ be two-dimensional boxes with the same center as $B_d$. Fix two vertices $a,b \in \frontA\slab{B}_d$, and define the events
\begin{align*}
    \mathcal{A}_{a,b}&=\left\{\begin{matrix}
        \exists \mbox{ open path }\gamma \mbox{ from }a \mbox{ to } \{0\}\times[0,t]\times[0,k];\\
        \exists \mbox{ open path }\gamma' \mbox{ from }b \mbox{ to } \{2t\}\times[0,t]\times[0,k];
        \\\gamma \cap \gamma' = \emptyset, (\gamma \cup \gamma')\cap \slab{B}_d=\emptyset
    \end{matrix}
    \right\},
    \\\mathcal{D}_{a,b}&= \{e \mbox{ piv. for } H(R_t)\}\cap \mathcal{A}_{a,b},
    \\\mathcal{G}_{a,b}&=\{f \mbox{ piv. for }a \stackrel{\slab{B}_d}{\longleftrightarrow}b\}\cap \{a \centernot \longleftrightarrow \S_k^+\setminus \slab{B}_d\}\cap \{b \centernot \longleftrightarrow \S_k^+\setminus \slab{B}_d\}.
\end{align*}

Fix $\omega \in \{e$ piv. for $H(R_t)\}$, let $C_{a,b}(\omega)$ be the union of the clusters of $a$ and $b$ inside $\slab{B}_d$ in $\omega$, and $E_{a,b}:=\{\langle u,v \rangle\in \mathcal{E}(\slab{B}_d):u =a \mbox{ or }u=b, v \notin \slab{B}_d\}$. Finally, for a pair of configurations $(\omega, \omega')\in \Omega \times \Omega$, define
$$\mathcal{K}(\omega, \omega')(e'):=\begin{cases}0&\mbox{ if }e'=e \mbox{ and }e \notin C_{a,b}(\omega'),\\
\omega'(e') &\mbox{ if }e' \notin E_{a,b} \mbox{ and at least one endpoint of }e' \mbox{ lies in }C_{a,b}(\omega'),\\
\omega(e') &\mbox{ otherwise.}\end{cases}$$

\begin{claim}\label{claim1} If $(\omega, \omega') \in \mathcal{D}_{a,b}\times \mathcal{G}_{a,b}$, then $\mathcal{K}(\omega, \omega') \in \{f \mbox{ piv. for }H(R_t)\}$.
\end{claim}

\begin{proof}
Indeed, in the new configuration $\mathcal{K}(\omega,\omega')$, the cluster $C_{a,b}$ is the same as in $\omega'$, except for the edges in $E_{a,b}$. Since we close the edge $e$ if it does not belong to $C_{a,b}(\omega')$, and $e$ is pivotal to $H(R_t)$ in $\omega$, there are no horizontal crossings of $R_t$ in $\mathcal{K}(\omega,\omega')$ without passing through $a$ and $b$. Also, since $f$ is pivotal for the connection of $a$ and $b$ inside $\slab{B}_d$, we find that $f$ is pivotal for $H(R_t)$ in $\mathcal{K}(\omega,\omega')$. 
\end{proof}

Claim \ref{claim1} allows us to conclude (as in Inequality (8) in \cite{Brochette}) that 
\begin{equation}\label{CotaFPiv}\P_{p,q}^\Lambda(f \mbox{ piv. for }H(R_t))\geq \P_{p,q}^\Lambda(\mathcal{K}(\mathcal{D}_{a,b}\times \mathcal{G}_{a,b}))\geq (1-p)\P_{p,q}^\Lambda(\mathcal{D}_{a,b})\P_{p,q}^\Lambda(\mathcal{G}_{a,b}).\end{equation}
Also, if $\omega \in \{e \mbox{ piv. for }H(R_t)\}$, there exist $a,b$ such that $\omega\in\mathcal{D}_{a,b}$. Hence,
\begin{equation}
\label{CotaDab}
    \P_{p,q}^\Lambda(\mathcal{D}_{a,b})\geq \frac{1}{|{\frontA\slab{ B}_d}|^2}\P_{p,q}^\Lambda(e\mbox{ piv. for }H(R_t))\geq\frac{1}{(8n^\lambda k)^2}\P_{p,q}^\Lambda(e\mbox{ piv. for }H(R_t)).
\end{equation}

\begin{claim}\label{claim2} There exist positive constants $c_{16}$ and $c_{17}$ such that, if $(\slab{B}_d\setminus \frontA\slab{B}_d)\cap (\Lambda\times \Z_+\times[0,k]) = \emptyset$, then for all $a,b \in \frontA\slab{B}_d$, $f \in \mathcal{E}(\frontA\slab{B}_d)$, and $(p,q) \in [p_c-(kn^{2\lambda})^{-1},p_c]\times(p_c,p_c+\varepsilon)$, it holds that  
\begin{equation}\label{claim_2}
\P_{p,q}^\Lambda(\mathcal{G}_{a,b})\geq c_{16}n^{-\lambda c_{17}}.
\end{equation}
\end{claim}
\begin{proof}
Observe that $\mathcal{G}_{a,b}$ is a cylinder event depending only on the edges of $\E(\slab{B}_{d+1})$, and let $\tilde{\omega}$ denote the restriction of $\omega$ to $\E(\slab{B}_{d+1})$. Since $|p-p_c|\leq (kn^{2\lambda})^{-1}$, $d\leq n^\lambda$, and $|\mathcal{E}(\slab{B}_{n^\lambda})|\leq c_{22}kn^{2\lambda}$ for some constant $c_{22}>0$, we find that
    \begin{align*}
        \P_{p,q}^\Lambda(\mathcal{G}_{a,b})&= \sum_{\tilde{\omega} \in \mathcal{G}_{a,b}}\P_{p,q}^\Lambda(\tilde{\omega})
        \\&=\sum_{\tilde{\omega} \in \mathcal{G}_{a,b}}\P^\Lambda_{p_c,q}(\tilde{\omega})\left(\tfrac{p}{p_c}\right)^{|\{\mbox{open edges of }\slab{B}_{d}\}|}\left(\tfrac{1-p}{1-p_c}\right)^{|\{\mbox{closed edges of }\slab{B}_{d}\}|}
        \\&\leq\sum_{\tilde{\omega} \in \mathcal{G}_{a,b}}\P^\Lambda_{p_c,q}(\tilde{\omega})\left[1+\tfrac{1}{(1-p_c)kn^{2\lambda}}\right]^{|\slab{B}_{d}|}\leq \sum_{\tilde{\omega} \in \mathcal{G}_{a,b}}e^{\tfrac{c_{22}}{1-p_c}}\P^\Lambda_{p_c,q}(\tilde{\omega})
        \\&=e^{\frac{c_{22}}{1-p_c}}\P_{p_c,q}^\Lambda(\mathcal{G}_{a,b}).
    \end{align*}
Therefore, we may assume that $p=p_c$ in the remainder of the proof.

\begin{figure}[ht]
         \centering
         \resizebox{0.6\linewidth}{!}{
             \begin{tikzpicture}

                \draw[very thick, color=black] (2,8)--(11,8);
                \draw[thick, color=black] (2,7)--(11,7);
                \draw[thick,color=black] (2,2)--(10,2);
                \draw[thick,color=black] (10,2)--(10,1);

               \draw[very thick, color=black] (9,8)--(6,2);
                \draw[very thick, color=black] (9,8)--(8,2);
                
                \draw[very thick, color=black] (4,2) rectangle (8,4);
                \draw[very thick, color=black!50] (6,2) -- (6,4);

                \draw[color=black!0,fill=black!50,pattern=north west lines] (4,1.8) rectangle (6,2);
                \draw[color=black!0,fill=black!50,pattern=crosshatch] (6,1.8) rectangle (8,2);

                \draw[color=black!0,fill=black!50,pattern=north west lines] (4,4.2) rectangle (6,4);
                \draw[color=black!0,fill=black!50,pattern=crosshatch] (6,4.2) rectangle (8,4);

                \draw[color=black!0,fill=black!50,pattern=north east lines] (3.8,2) rectangle (4,4);
                
                \draw[color=black!0,fill=black!50,pattern=crosshatch] (8.5,7) rectangle (8.8,6.8);

                \fill[color=black!5,pattern=grid,opacity=0.2] (6,2)-- (8,2)--(8.75,7)--(8.5,7) -- cycle;
                \filldraw[black] (9,8) circle (2pt) node[anchor=south]{$a$};
                \node[align=left,color=black] at (5,1.5) {$K^2_1$};
                \node[align=left,color=black] at (7,1.5) {$K^1_1$};
                
                \node[align=left,color=black] at (4.5,3.5) {$R^2_1$};
                \node[align=left,color=black] at (6.4,3.5) {$R^1_1$};
                
                \node[align=left,color=black] at (8,4.7) {$\overline{C_1}$};
                \node[align=left,color=black] at (7.25,5.7) {$C_1$};
                
                \node[align=left,color=black] at (11.5,8) {$B_d$};  
                \node[align=left,color=black] at (11.5,7) {$B''_d$};  
                \node[align=left,color=black] at (10.5,2) {$B'_d$};  

                \node[align=left,color=black] at (5.3,4.5) {$T^2_1$};                
                \node[align=left,color=black] at (6.9,4.5) {$T^1_1$};
                \node[align=left,color=black] at (3.4,3) {$S_1^{2,2}$};
                \node[align=left,color=black] at (8.6,3) {$S_1^{1,1}$};
                \node[align=left,color=black,scale=0.75] at (5.6,2.5) {$S_1^{2,1} $};
                \node[align=left,color=black,scale=0.75] at (6.6,2.5) {$S_1^{1,2} $};
                
                \node[align=left,color=black] at (8.45,6.3) {$C^*_1$};

             \end{tikzpicture}
             }
     \caption{The sets $C_1$, $R_1^1$, $R_1^2$ and their subsets.}
     \label{fig:WeirdSets}
 \end{figure}
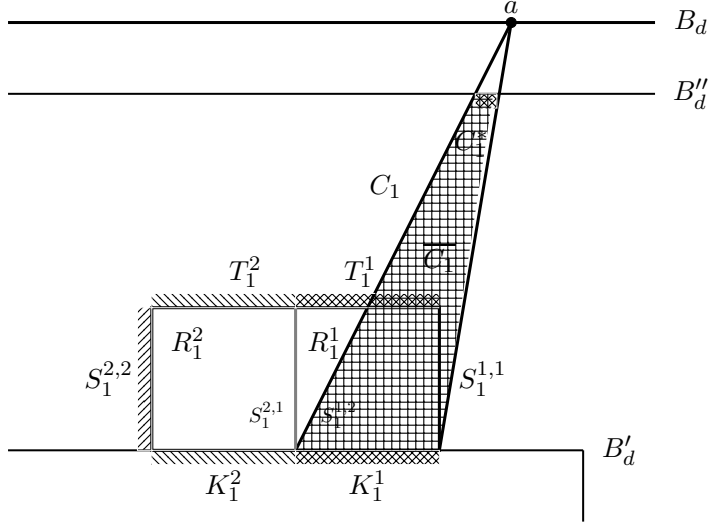

Fix $a,b \in \frontA\slab{B}_{d}$ and consider the edge $f=\langle c,d \rangle \in \mathcal{E}(\frontA\slab{B}_d)$, as fixed at the beginning of the proof. Assume that $a, b, c, d$ are all distinct (the remaining cases are similar). Divide each side of $B_d'$ into 25 intervals of length $d/25$. Select $11$ intervals $\{K_i\}_{i=1}^{11}$, numbered counter-clockwise, with the following properties:
\begin{enumerate}
    \item no interval contains any of the corner vertices of $B_d'$,
    \item the intervals are non-adjacent,
    \item the intervals $K_1$, $K_2$, and $K_3$ are in the same side as the vertex $a$,
    \item the intervals $K_4$, $K_5$, $K_6$, $K_7$, and $K_8$ are in the same side as the edge $f$,
    \item the intervals $K_9$, $K_{10}$, and $K_{11}$ are in the same side as the vertex $b$.
\end{enumerate}

For every $i = 1,\dots,11$, divide the interval $K_i$ into halves and label each half $K_i^1,K_i^2$ in a counter-clockwise fashion. Define for all $i=1,...,11$ and $j=1,2$, the rectangles $R^j_i$ of side $K_i^j$ and height $\frac{d}{100}$, contained in $B''_d\setminus B'_d$, whose side opposite to $K_i^j$ we call $T_i^j$. Let the sides of each rectangle $R_i^j$ be called $S_i^{j,1}$ and $S_i^{j,2}$, respectively, again numbered counter-clockwise. 

For $i=1,...,3$, let $C_i$ be the cone joining $K^1_i$ to $a$, for $i=4,...,8$, let $C_i$ be the cone joining $K_i^1$ to $c$, and for $i=9,10,11$, let $C_i$ be the cone joining $K_i^1$ to $b$. For all $i=1,..,11$, let $\overline{C}_i = C_i \cap B''_d$ and $C^*_d = C_i \cap \frontB B''_d$. We refer the reader to Figure \ref{fig:WeirdSets} for a sketch of the sets above.

Lemma \ref{GluingLemmaGenerico}, together with Theorem 3.1 of \cite{tassion} and standard Russo-Seymour-Welsh techniques, allow us to show the existence of a constant $c_{23}>0$ such that
\begin{equation}
    \label{eq:Polinomiocone1}
    \P_{p_c,q}^\Lambda(\slab{C}_i^*\stackrel{\slab{\overline{C}}_i \cup\slab{R}^1_i\cup\slab{R}^2_i}{\looooongleftrightarrow}\slab{S}_i^{2,2})>d^{-c_{23}}.
    \end{equation}

    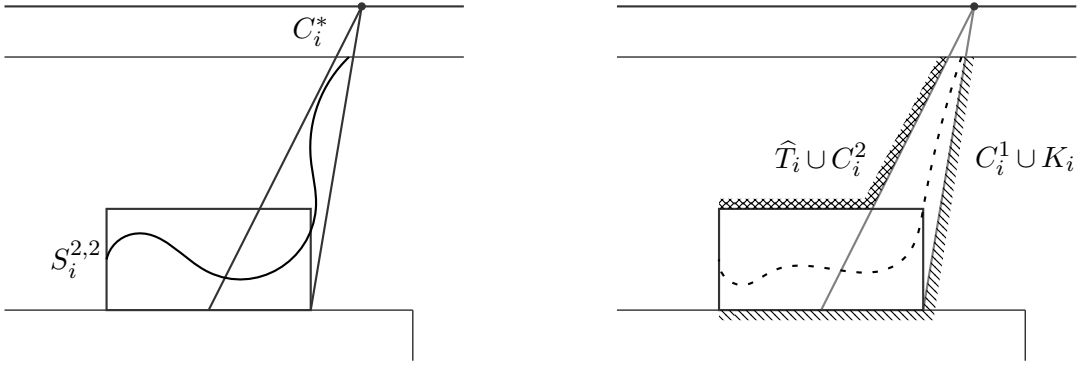
\begin{figure}[ht]
         \centering
         \resizebox{0.9\linewidth}{!}{
             \begin{tikzpicture}
                \draw[very thick, color=black!80] (2,8)--(11,8);
                \draw[thick, color=black!80] (2,7)--(11,7);
                \draw[thick,color=black!80] (2,2)--(10,2);
                \draw[thick,color=black!80] (10,2)--(10,1);
                \draw[very thick, color=black!80] (9,8)--(6,2);
                \draw[very thick, color=black!80] (9,8)--(8,2);
                \draw[very thick, color=black!80] (4,2) rectangle (8,4);
                \filldraw[black!80] (9,8) circle (2pt);
                \draw [very thick,color=black] (8.75,7) to [ curve through ={(8,5)..(8.1,4)..(6,2.75)..(4.5,3.5)}] (4,3);
                \node[align=left,color=black,scale=1.5] at (3.4,3) {$S_i^{2,2}$};
                \node[align=left,color=black,scale=1.5] at (8,7.5) {$C^*_i$};

                \draw[very thick, color=black!80] (14,8)--(23,8);
                \draw[thick, color=black!80] (14,7)--(23,7);
                \draw[thick,color=black!80] (14,2)--(22,2);
                \draw[thick,color=black!80] (22,2)--(22,1);
                \draw[very thick, color=black!50] (21,8)--(18,2);
                \draw[very thick, color=black!50] (21,8)--(20,2);
                \draw[very thick, color=black!80] (16,2) rectangle (20,4);
                \fill[color=black!80,pattern=north west lines] (16,1.8)-- (20.2,1.8)--(21,7)--(20.8,7)--(20,2)--(16,2) -- cycle;
                \fill[color=black!80,pattern=crosshatch] (16,4) --(19,4)--(20.5,7)--(20.3,7) -- (18.85,4.2) -- (16,4.2) --  cycle;
                \filldraw[black!80] (21,8) circle (2pt);
                \node[align=left,color=black,scale=1.5] at (18,5) {$\widehat{T}_i\cup C_i^2$};
                \node[align=left,color=black,scale=1.5] at (22,5) {$C_i^1\cup K_i$};
                \draw [very thick,color=black,loosely dashed] (20.75,7) to [ curve through ={(20,4)..(19.6,3)..(17,2.75)..(16.5,2.5)}] (16,3);
             \end{tikzpicture}
             }
     \caption{On the left, a two-dimensional sketch of the open path connecting $\slab{C}^*_i$ to $\slab{S}_{i}^{2,2}$. On the right, the closed cutset in dashed lines, separating the marked regions.}
     \label{fig:CrossingOpenAndClosed}
 \end{figure}

    Let $C_i^1$ and $C_i^2$ be the sides of the cone $\overline{C}_i$, numbered in a counter-clockwise fashion. Denote by $C^{j,\uparrow}_i = C^j_i\setminus(R_i^1\cup R_i^2)$ and let $\widehat{T}_i \subseteq (T_i^1\cup T_i^2)\setminus \overline{C}_i$ be the top of the rectangle $R_i^1\cup R_i^2$ adjacent to $C_i^2$ (see Figure \ref{fig:CrossingOpenAndClosed}).

    The same construction used to obtain \eqref{eq:Polinomiocone1} can be applied to show the existence of a closed cutset inside $\slab{\overline{C}}_i\cup \slab{R}_i^1\cup \slab{R}_i^2$ that separates $\slab{K}_i \cup \slab{C}_i^1$ and $\slab{\widehat{T}}_i \cup \slab{C}_i^2$. This shows that 
\begin{equation*}
\label{eq:Polinomiocone2}
    \P_{p_c,q}^\Lambda(\slab{K}_i \cup \slab{C}_i^1\stackrel{\slab{\overline{C}}_i \cup\slab{R}^1_i\cup\slab{R}^2_i}{\centernot\looooongleftrightarrow}\slab{\widehat{T}}_i \cup \slab{C}_i^2)>d^{-c_{24}}.    \end{equation*}

We proceed by constructing two disjoint paths in $\slab{B}'_d\cup \left(\bigcup_{i =1}^{11}(\slab{R}_i^1\cup \slab{R}_i^2\cup \overline{\slab{C}}_i)\right)$ joining $\slab{C}_2^*$ to $\slab{C}_5^*$ and $\slab{C}_7^*$ to $\slab{C}_{10}^*$, respectively. A second application of the Russo-Seymour-Welsh technique, combined with Lemma \ref{GluingLemmaGenerico}, ensures that these paths exist and are disjoint from one another and from $\frontB \slab{B}''_d$ with probability larger than $d^{-c_{25}}$, for some constant $c_{25}>0$. Let $\Gamma^*$ be the event that these paths exist.
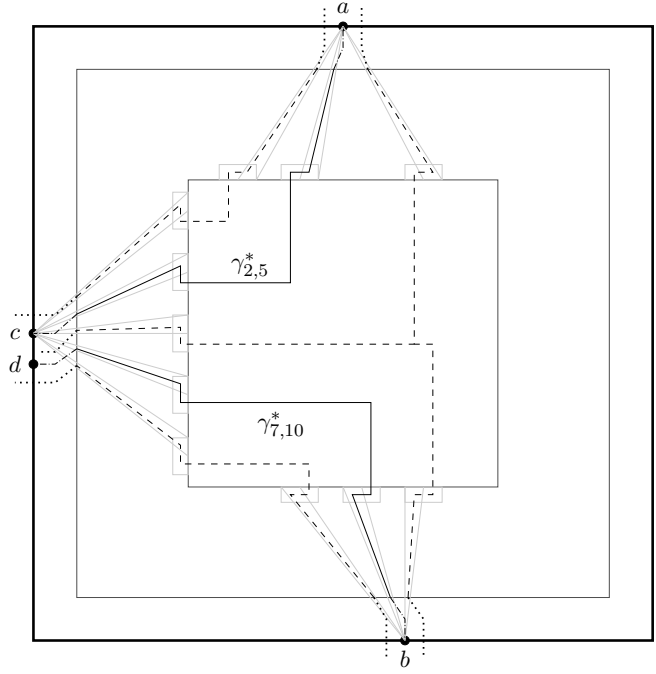
\begin{figure}[!ht]
         \centering
         \resizebox{0.55\linewidth}{!}{
             \begin{tikzpicture}
                \draw[very thick, color=black] (0,0) rectangle (10,10);
                \draw[color=black!70] (0.7,0.7) rectangle (9.3,9.3);
                \draw[color=black!70] (2.5,2.5) rectangle (7.5,7.5);

                \filldraw[black] (5,10) circle (2pt); 
                \node[align=left,black] at (5,10.3) {$a$};
                
                \filldraw[black] (0,5) circle (2pt);
                \node[align=left,black] at (-0.3,5) {$c$};
                
                \filldraw[black] (0,4.5) circle (2pt); 
                \node[align=left,black] at (-0.3,4.5) {$d$};

                \filldraw[black] (6,0) circle (2pt); 
                \node[align=left,black] at (6,-0.3) {$b$};

                \draw[color=black!20] (3,7.5) rectangle (3.6,7.75); \draw[color=black!20] (5,10) -- (3.6,7.5);\draw[color=black!20] (5,10) -- (3.3,7.5);
                \draw[color=black!20] (4,7.5) rectangle (4.6,7.75); \draw[color=black!20] (5,10) -- (4.6,7.5);\draw[color=black!20] (5,10) -- (4.3,7.5);
                \draw[color=black!20] (6,7.5) rectangle (6.6,7.75); \draw[color=black!20] (5,10) -- (6.6,7.5);\draw[color=black!20] (5,10) -- (6.3,7.5);

                \draw[color=black!20] (2.25,2.7) rectangle (2.5,3.3);\draw[color=black!20] (0,5) -- (2.5,3.3);\draw[color=black!20] (0,5) -- (2.5,3);
                \draw[color=black!20] (2.25,3.7) rectangle (2.5,4.3);\draw[color=black!20] (0,5) -- (2.5,4.3);\draw[color=black!20] (0,5) -- (2.5,4);
                \draw[color=black!20] (2.25,4.7) rectangle (2.5,5.3);\draw[color=black!20] (0,5) -- (2.5,5.3);\draw[color=black!20] (0,5) -- (2.5,5);
                \draw[color=black!20] (2.25,5.7) rectangle (2.5,6.3);\draw[color=black!20] (0,5) -- (2.5,6.3);\draw[color=black!20] (0,5) -- (2.5,6);
                \draw[color=black!20] (2.25,6.7) rectangle (2.5,7.3);\draw[color=black!20] (0,5) -- (2.5,7.3);\draw[color=black!20] (0,5) -- (2.5,7);

                \draw[color=black!20] (4,2.25) rectangle (4.6,2.5);\draw[color=black!20] (6,0) -- (4,2.5);\draw[color=black!20] (6,0) -- (4.3,2.5);
                \draw[color=black!20] (5,2.25) rectangle (5.6,2.5);\draw[color=black!20] (6,0) -- (5,2.5);\draw[color=black!20] (6,0) -- (5.3,2.5);
                \draw[color=black!20] (6,2.25) rectangle (6.6,2.5);\draw[color=black!20] (6,0) -- (6,2.5);\draw[color=black!20] (6,0) -- (6.3,2.5);

                \draw [color=black,dashed] (4.575,9.3)-- (3.45,7.625)--(3.15,7.625)--(3.15,6.825)--(2.375,6.825)--(2.375,7.1)-- (0.7,5.6);
                \draw [color=black] (4.85,9.3)-- (4.45,7.625)--(4.15,7.625)--(4.15,5.825)--(2.375,5.825)--(2.375,6.1)-- (0.7,5.315);
                \draw [color=black,dashed] (5.4,9.3)-- (6.45,7.625)--(6.15,7.625)--(6.15,4.825)--(2.375,4.825)--(2.375,5.1)-- (0.7,5.05);

                \draw [color=black,dashed] (6.05,0.7)-- (6.15,2.375)-- (6.45,2.375)-- (6.45,4.825) --(6.15,4.825);
                \draw [color=black] (5.765,0.7)-- (5.15,2.375)-- (5.45,2.375)-- (5.45,3.875) --(2.375,3.875) -- (2.375,4.18) -- (0.7,4.75);
                \draw [color=black,dashed] (5.5,0.7)-- (4.15,2.375)-- (4.45,2.375)-- (4.45,2.875) --(2.375,2.875) -- (2.375,3.18) -- (0.7,4.475);
                
                \draw[color=black,thick, dotted](5.5,0.7) -- (5.7,0.35) -- (5.7,-0.3);
                \draw[color=black,thick,dotted](6.05,0.7) -- (6.3,0.35) -- (6.3,-0.3);
                \draw[color=black,densely dashdotted] (5.765,0.7)--(6,0.35)--(6,0);

                \draw[color=black,thick, dotted](4.575,9.3) -- (4.7,9.65) -- (4.7,10.3);
                \draw[color=black,thick, dotted](5.4,9.3) -- (5.3,9.65) -- (5.3,10.3);
                \draw[color=black,densely dashdotted] (4.85,9.3)--(5,9.65)--(5,10);

                \draw[color=black,thick, dotted](0.7,5.6) -- (0.35,5.3) -- (-0.3,5.3);
                \draw[color=black,densely dashdotted] (0.7,5.315)--(0.35,5)--(0,5);
                \draw[color=black,thick,  dotted](0.7,5.05) -- (0.35,4.7) -- (0.1,4.7);
                \draw[color=black,densely dashdotted] (0.7,4.75)--(0.35,4.5)--(0,4.5);
                \draw[color=black,thick, dotted](0.7,4.475) -- (0.35,4.2) -- (-0.3,4.2);

                \node[align=left,scale=1] at (4,3.5){$\gamma^*_{7,10}$};
                \node[align=left,scale=1] at (3.5,6.1){$\gamma^*_{2,5}$};

            \end{tikzpicture}
             }

    \caption{Two-dimensional sketch of the paths constructed in the argument. The paths $\gamma^*_{2,5}$ and $\gamma^*_{7,10}$ are in continuous lines, the closed cutsets in dashed lines. In dotted lines are the local configurations of closed edges that ensure pivotality of the edge $f = \langle c,d\rangle$, and in dash-dotted lines are the local configurations of open edges.}
    \label{fig:BigCrossings}
\end{figure}

By choosing a configuration $\omega \in \Gamma^*$ and changing the state of the edges on $\slab{B}_d\setminus \slab{B}''_d$ around the vertices $a, b, c$ and $d$, we can create a path $\gamma^*_{2,5}$ joining $a$ and $c$, and a path $\gamma^*_{7,10}$ joining $b$ and $d$. These new paths ensure that the edge $f= \langle c,d \rangle$ becomes pivotal for the connection between $a$ and $b$ inside $\slab{B}_d$ (see Figure \ref{fig:BigCrossings}). Therefore, we obtain 
\begin{equation*}
\label{CotaGab}
    \P_{p,q}^\Lambda(\mathcal{G}_{a,b})\geq e^{\frac{c_{22}}{1-p_c}}\P_{p_c,q}^\Lambda(\mathcal{G}_{a,b})\geq c_{26}\P_{p_c,q}(\Gamma^*)\geq c_{26}d^{-c_{25}}\geq c_{26}n^{-\lambda c_{25}},
\end{equation*}
thus proving \eqref{claim_2} with $c_{16} = c_{26}$ and $c_{17} = c_{25}$.
\end{proof}

Plugging \eqref{CotaDab} and \eqref{claim_2} into \eqref{CotaFPiv}, we obtain
\begin{equation}\label{eq:ComparisonPivotality}\P_{p,q}^\Lambda(f \mbox{ piv. for }H(R_t))\geq c_{18} (n^{\lambda})^{-c_{19}}\P_{p,q}^\Lambda(e \mbox{ piv. for }H(R_t)).\end{equation}

For each  $f \in Enh(\Lambda)$, write $$\Psi(f)\coloneqq\{e \in \E(\slab{R}_t): f \mbox{ is the closest enhanced edge to }e\}.$$
Applying Russo's formula and using \eqref{eq:ComparisonPivotality}, we find that
\begin{align*}
    \frac{\partial }{\partial p}\P_{p,q}^\Lambda (H(R_t)) &= \sum_{e \in \E(\slab{R}_t)\setminus Enh(\Lambda)}\P_{p,q}^\Lambda(e \mbox{ piv. for }H(R_t))
    \\&=\sum_{f \in Enh(\Lambda)}\sum_{e \in \Psi(f)}\P^\Lambda_{p,q}(e \mbox{ piv. for }H(R_t))
    \\&\leq \sum_{f \in Enh(\Lambda)} |\Psi(f)| \max_{e \in \Psi(f)}\left\{\P_{p,q}^\Lambda(e \mbox{ piv. for }H(R_t))\right\}
    \\&\leq \sum_{f \in Enh(\Lambda)}c_{20}kn^{2\lambda}\cdot \frac{n^{\lambda c_{19}}}{c_{18}}\P_{p,q}^\Lambda(f \mbox{ piv. for }H(R_t))
    \\&=c_{21}(n^{\lambda})^{c_{15}}\frac{\partial}{\partial q}\P_{p,q}^\Lambda(H(R_t)),
\end{align*}
where $c_{15} = 2+c_{19}$.

Let $h>0$ be a constant to be fixed later, and consider $p \in [p_c-(kn^{2\lambda})^{-1},p_c)$ so that \eqref{claim_2} holds. Write $p(t) = tp+(1-t)(p_c+h)$ and $q(t) = tq+(1-t)(p_c+h)$ to obtain
\begin{small}
\begin{align*}
    \P_{p,q}^\Lambda(H(R_t))-\P_{p_c+h}^\Lambda(H(R_t)) &= \int_0^1 \frac{d}{dt}\P^{\Lambda}_{p(t),q(t)}(H(R_t))dt
    \\&=\int_0^1 \left[\frac{\partial }{\partial p}\P_{p,q}^\Lambda(H(R_t))p'(t)+\frac{\partial }{\partial q}\P_{p,q}^\Lambda(H(R_t))q'(t)\right]dt
    \\&=\int_0^1 \left[\frac{\partial }{\partial p}\P_{p,q}^\Lambda(H(R_t))(p-(p_c+h))+\frac{\partial }{\partial q}\P_{p,q}^\Lambda(H(R_t))(q-(p_c+h))\right]dt
    \\&\geq \int^1_0 \left\{\left[(p-(p_c+h))c_{21}(n^\lambda)^{c_{15}}+q-(p_c+h)\right]\frac{\partial}{\partial q}\P_{p,q}^\Lambda(H(R_t))\right\}dt,
\end{align*}
\end{small}
where the inequality holds since $p'(t)<0$.

Let $\delta=\frac{\varepsilon}{2(1+c_{21}n^{\lambda c_{15}})}$ and take $h = c_{14}n^{-\lambda c_{15}}$, with $c_{14}<\frac{\varepsilon}{3(1+c_{21})}$. Setting  $p = p_c-\delta$ and $q = p_c+\varepsilon$, we have 
$$(p-(p_c+h))c_{21}(n^\lambda)^{c_{15}}+q-(p_c+h)>0,$$ and hence
$$\P^\Lambda_{p_c-\delta, p_c+\varepsilon}(H(R_t))\geq \P_{p_n}(H(R_t)),$$
yielding the lemma.
\end{proof}

\begin{lema}\label{piecesHighProbability} Let $p_n$ be as in Lemma \ref{aizenmanGrimmettStep}. For all $\vartheta>0$, there exists $\lambda_0 = \lambda_0(\vartheta)>0$ such that, for every $\lambda<\lambda_0$, it holds that
$$\min\left\{\P_{p_n}(H_n), \,\P_{p_n}(V_n), \,\P_{p_n}(D_n)\right\}>1-\vartheta,$$
for $n$ sufficiently large.
\end{lema}
Before we proceed to a formal proof of Lemma \ref{piecesHighProbability}, we outline the main ideas. First, Theorem \ref{corr_length_2} shows that, with an appropriate choice of $\lambda$, one can take $n$ large enough so that $n/2$ exceeds the correlation length at $p_n$, implying that the rectangle $[0,n] \times [0,n/2] \times [0,k]$ is crossed with high probability. Since $n$ also exceeds the correlation length, this also implies that the rectangle $[0,2n]\times[0,n]\times [0,k]$ is crossed with high probability. Second, with the aid of Proposition 3.11 of \cite{tassion}, one can show that the events $H_{n/2}$, $V_{n/2}$, $H_n$, and $V_n$ have high probability. Finally, Lemma \ref{GluingLemmaGenerico}, combined with crossings of $H_{n/2}$ and $V_{n/2}$, yields a bound for $\P_{p_n}(D_n)$. The construction is carried out backwards in order to determine the appropriate size of the correlation length, depending on $\vartheta$.
\begin{proof}[Proof of Lemma \ref{piecesHighProbability}]
    By Lemma \ref{GluingLemmaGenerico}, there exists $\vartheta_1>0$ such that, if 
    \begin{equation}\label{eq:CotaDN}
    \P_{p_n}\left(H\left(\left[0,4\left(\frac{n}{2}\right)\right]\times\left[0,\frac{n}{2}\right]\times[0,k]\right)\right)>1-\vartheta_1,\end{equation}
    then $\P_{p_n}(D_n)>1-\vartheta$.
According to Proposition 3.11 of \cite{tassion}, there exists $\vartheta_2>0$ such that, if 
    \begin{equation}
        \label{eq:UsoCorrLen}
\P_{p_n}\left(H([0,2n]\times\left[0,n\right]\times[0,k])\right)>1-\vartheta_2,    \end{equation}
    then 
    \begin{align*}
        \P_{p_n}(H_n) &=  \P_{p_n}(V_n) =\P_{p_n}(H([0,4n]\times\left[0,n\right]\times[0,k]))
        \\&>\max\{1-\vartheta,1-\vartheta_1\}.
        \end{align*}
        
Applying Theorem \ref{corr_length_2} with $\tau=\vartheta_2$, we obtain 
        \begin{align*}
           L_{\vartheta_2}(p_n)= L_{\vartheta_2}(p_c+c_{14}n^{-\lambda c_{15}})&\leq c_1 c_{14}n^{\lambda c_{15}c_2(\vartheta_2)}.
        \end{align*}
        Therefore, if $\lambda<\frac{1}{c_{15}c_2(\vartheta_2)}=:\lambda_0$, then for some sufficiently large $n_0$, we have $\frac{n_0}{2}\geq c_0L_{\vartheta_2}(p_{n_0})$, for some constant $c_0>0$. Consequently, if $n>n_0$, both inequalities \eqref{eq:CotaDN} and \eqref{eq:UsoCorrLen} hold. This completes the proof.
        \end{proof}

\begin{proof}[Proof of Proposition \ref{theoremEdge}] Fix $\constNossoBeta<1$ and let $\vartheta = \vartheta(
\rho)>0$ be such that $h_3(1-\vartheta)>\constNossoBeta$, where $h_3$ is the function defined in Lemma \ref{gluingLemmaEdge}. Take $\lambda_0 = \lambda_0(\vartheta)$ as given by Lemma \ref{piecesHighProbability}, fix $\lambda<\lambda_0$ and take $\delta_n(\varepsilon,n,\lambda)$ as in Lemma \ref{aizenmanGrimmettStep}. With these choices, for a given $\lambda$-favored edge $f$, we obtain
 \begin{align*}
        \P^{\Lambda}_{p_c-\delta_n,p_c+\varepsilon}(A_n(f))&\geq  h_3(\P^\Lambda_{p_c-\delta_n,p_c+\varepsilon}(D_n)\wedge \P^\Lambda_{p_c-\delta_n,p_c+\varepsilon}(V_n)\wedge \P^\Lambda_{p_c-\delta_n,p_c+\varepsilon}(H_n))
        \\&\geq h_3(\P^\Lambda_{p_n}(D_n)\wedge \P^\Lambda_{p_n}(V_n)\wedge \P^\Lambda_{p_n}(H_n))
        \\&\geq h_3(1-\vartheta)>\constNossoBeta.
    \end{align*}
    The first inequality follows from Lemma \ref{gluingLemmaEdge}, which is independent of the fixed constants. The second inequality holds because the edge $f$ is $\lambda$-favored, ensuring that the events $D_n$, $V_n$, and $H_n$ occur within $\lambda$-good blocks. Finally, the third inequality follows from Lemma \ref{piecesHighProbability} by taking $n$ sufficiently large, since $\lambda<\lambda_0$.
    \end{proof}
\section{Multiscale scheme}\label{secMulti}
In this section, we conclude the proof of Theorem \ref{mainTheorem}. So far, we have constructed a 1-dependent bond percolation model in $\Z_+^2$ defined by the block construction in \eqref{BlockConstruction}.  In this model, favored edges are open with probability arbitrarily close to one, while unfavored edges are open with positive probability. In Section \ref{SectionDominacao}, we show that this 1-dependent percolation model dominates an inhomogeneous independent percolation process on $\Z_+^2$, thereby establishing Proposition \ref{TeoDominacao}. 

 In the independent model mentioned above, all edges are open with high probability, except those incident to randomly selected columns (depending on the original environment $\Lambda$), which are open with positive probability. In Section \ref{SectionMultiscale}, we apply a multiscale scheme based on the approach of \cite{Marcos}, to show that this model percolates.



\subsection{Proof of Proposition \ref{TeoDominacao}}
\label{SectionDominacao}

Let $B_m \subset \Z_+^2$. Fix an ordering $\{f_1, \dots, f_{|\E(B_m)|}\}$ of the renormalized edges in $\E(B_m)$ such that, for some $J$, the edge $f_i$ is favored for $i=1,\dots, J$ and unfavored otherwise. We shall prove Proposition \ref{TeoDominacao} by induction on the number of edges, using the proof strategy of Liggett-Schonmann-Stacey theorem \cite{LSS} for favored edges and employing local modifications alongside the FKG inequality for unfavored edges.

Proposition 1.2 of \cite{LSS} shows the existence of $\constNossoBeta_1<1$ and a function $g:[0,1] \to [0,1]$, with $\lim_{x \to 1}g(x)=1$, such that for $\rho>\constNossoBeta_1$, whenever $\mathbb{Q}^{\Lambda,n}_{p,q}(W_{f_i}=1)>\rho$ for $0 \leq i \leq j$, then 
\begin{equation}\label{eq:DesigLSS}    \Q^{\Lambda,n}_{p,q}\left(W_{f_{j}}=1\Big|W_{f_i}=w_i \mbox{ for }1\leq i < j\right)\geq g(\constNossoBeta),
\end{equation}
for every $w_i \in\{0,1\}$, $i<j$. Since all favored edges appear first in the ordering of $\E(B_m)$, this argument shows that \eqref{eq:DesigLSS} holds for every favored edge $f_j$ with $j\leq J$.


After the set of favored edges is exhausted, we begin to consider unfavored edges. Our goal is to find a constant $c_3(n,\constNossoMu)$ such that, for each unfavored edge $f_{j}$, with $j>J$,
\begin{equation}\label{eq: dominacaoEloRuim}\Q^{\Lambda,n}_{p,q}\left(W_{f_{j}}=1\Big|W_{f_i}=w_i \mbox{ for }i\neq j \right)\geq c_3(1-\constNossoMu)\P^{\Lambda}_{p,q}(\sigma_n(f_{j})=1).\end{equation}

Once \eqref{eq: dominacaoEloRuim} is proven, the desired domination follows by associating independently to each $f \in \E(\mathbb{Z}_+^2)$ a uniform random variable $U_f\sim U[0,1]$, and setting 
$$Y_{f} = \left\{\begin{array}{lll}
    1& \mbox{ if }U_f<g(\constNossoBeta) \mbox{ and }f \mbox{ is favored,}\\
    1& \mbox{ if }U_f<c_3(1-\constNossoMu)\P^{\Lambda}_{p,q}(\sigma_n(f)=1) \mbox{ and } f \mbox{ is unfavored,}
    \\0&\mbox{ otherwise}.&
\end{array}\right.$$


We now prove \eqref{eq: dominacaoEloRuim}. Let $f = \langle x,y \rangle$ be an unfavored edge, and denote by $\varphi_1,...\varphi_6$ the edges incident to $f$. Let $\chi = (\chi_1,...,\chi_6) \in \{0,1\}^6$ be a configuration on these edges. Observe that $\{\sigma_n(f)\}_{f \in \E(B_m)}$, defined in \eqref{BlockConstruction}, depends solely on the states of the edges of $\E(\slab{B}_{(2m+1)n})$ in $\S_k^+$. We shall prove \eqref{eq: dominacaoEloRuim} by conditioning on the elements of $\E(\slab{B}_{(2m+1)n})$ outside the set $\slab{B}_n(x)\cup \slab{B}_n(y)$. To do so, fix $\kappa \in \{0,1\}^{\E(\slab{B}_{(2m+1)n})\setminus \E(\slab{B}_n(x)\cup \slab{B}_n(y))}$ and define
$$P_{\kappa}(\cdot) \coloneqq \P^{\Lambda}_{p,q}\left(\cdot \,| \omega(e) = \kappa(e), \, \forall\, e \in \E(\slab{B}_{(2m+1)n})\setminus \E(\slab{B}_n(x)\cup \slab{B}_n(y))\right).$$
We will show that 
$$P_{\kappa}(W_{f}=1|W_{\varphi_i} = \chi_i, \, \forall\, i=1,...,6) \geq (1-\constNossoMu)c_{3}\mathbb{Q}^{\Lambda,n}_{p,q}(\sigma_n(f)=1),$$
for each possible $(\chi,\kappa) \in \{0,1\}^6\times\{0,1\}^{\E(\slab{B}_{(2m+1)n})\setminus \E(\slab{B}_n(x)\cup \slab{B}_n(y))}$. The desired result will follow by averaging over $\kappa$.

Letting $I_{\chi} = \{i \in \{1,...,6\}: \chi_i=1\}$, we obtain
\begin{align*}
    P_{\kappa}&\left(W_{f}=1\cap \left\{\sigma_n(\varphi_i)Z_{\varphi_i}=\chi_i,\, \forall\, i = 1,...,6\right\}\right)\geq
    \\&\geq P_{\kappa}\left(W_{f}=1\cap\left\{(\sigma_n(\varphi_i)Z_{\varphi_i}=\chi_i)\cap (Z_{\varphi_i}=\chi_i),\, \forall\, i = 1,...,6\right\}\right)
    \\&= P_{\kappa} \left(W_{f}=1\cap \left\{\sigma_n(\varphi_i)=1,\, \forall\, i \in I_{\chi}\right\}\cap \left\{Z_{\varphi_i}=\chi_i, \, \forall\, i=1,...,6\right\}\right)
    \\&\geq \constNossoMu^6 P_{\kappa} \left(W_{f}=1\cap \left\{\sigma_n(\varphi_i)=1,\, \forall i \in I_{\chi}\right\}\right).
\end{align*}

Recall the definition of  $C(f)$ from \eqref{front}. For each $i = 1,..., 6$, define the event $A^f_n(\varphi_i):=\{\sigma_n(\varphi_i)=1\}\cap \{\omega(e)=1: e \in {C}(f)\}$. Since $f$ is unfavored, $\{\sigma_n(f)=1\} = A^*_n(f)$, which is an increasing event. Conditioned on $\kappa$, the events $\{A^f_n(\varphi_i)\}_{i\leq 6}$ are also increasing. Therefore, applying the FKG inequality, we have 
\begin{align}
    \constNossoMu^6 P_{\kappa} \left(W_{f}=1\cap \left\{\sigma_n(\varphi_i)=1,\, \forall\, i \in I_{\chi}\right\}\right)&\geq \constNossoMu^6 P_{\kappa} \left(W_{f}=1\cap \left\{A^f_n(\varphi_i),\, \forall\, i \in I_{\chi}\right\}\right) \nonumber
    \\&\geq \constNossoMu^6 P_{\kappa} \left(W_{f}=1\right) P_{\kappa}\left(A^f_n(\varphi_i),\, \forall\, i \in I_{\chi}\right).
    \label{eq:DesigIntermediariaDominacao}
\end{align}

Let $F = \{\sigma_n(\varphi_i)=1,\,\forall\, i \in I_{\chi}\}$. For every $\omega \in F$, define the function 
\begin{align*}
  g \colon &F \to F\\
  &\omega \mapsto g_{\omega},
\end{align*}
with
$$g_{\omega}(e) = \left\{\begin{matrix}
    \omega(e) &\mbox{ if } e \notin {C}(f),
    \\1 &\mbox{ if }e \in {C}(f).
\end{matrix}\right.$$


Given that $\frac{1}{5}\leq p_c(\Z^3)<p_c(\S_k^+)\leq\frac{1}{2}$, we can assume $\frac{1}{5}< q<p<\frac{1}{2}$. This assumption implies
$$ \max\left\{\tfrac{1-p}{p}, \tfrac{p}{1-p}, \tfrac{q}{1-q}, \tfrac{1-q}{q}\right\}\leq4.$$
Observing that $|\{e \in \E(\slab{B}_{(2m+1)n}):\omega(e)\neq g_{\omega}(e)\}|\leq 8nk$, we conclude that, for all $\omega \in F$, 
\begin{align*}
    P_{\kappa}(\omega)\leq
    4^{8nk}P_\kappa(g_\omega).
\end{align*}
Writing $g_F=\{g(\omega): \omega\in F\}$, we obtain 
\begin{align*}
    P_{\kappa}(F)&=\sum_{\omega \in F}P_{\kappa}(\omega)\leq\sum_{\omega \in F}4^{8nk}P_\kappa(g_\omega)
    \\&= 4^{8nk} \sum_{\omega \in g_F}P_{\kappa}(\omega)|\{\omega': g_{\omega'}=\omega\}|
    \\&\leq (4^{8nk})(2^{8nk}) \sum_{\omega \in g_F}P_{\kappa}(\omega)
    \\&= 8^{8nk} P_{\kappa}(A^f_n(\varphi_i)\,\forall i \in I_{\chi}),
\end{align*}
which, combined with \eqref{eq:DesigIntermediariaDominacao}, gives
\begin{align*}
P_{\kappa}&\left(W_{f}=1\cap \left\{\sigma_n(\varphi_i)Z_{\varphi_i}=\chi_i,\, \forall\, i = 1,...,6\right\}\right)\geq
\\&\geq \constNossoMu^68^{-8nk} P_{\kappa}\left(W_{f}=1\right)P_{\kappa}(\sigma_n(\varphi_i)=1,\, \forall\, i \in I_{\chi})
\\&\geq \constNossoMu^68^{-8nk} P_{\kappa}\left(W_{f}=1\right)P_{\kappa}(W_{\varphi_i}=1,\, \forall\, i \in I_{\chi}).
\end{align*}
Letting $c_{3} = \constNossoMu^68^{-8nk}$ and observing that $\{W_{\varphi_i}=\chi_i,\, \forall\, i=1,...,6\}\subseteq \{W_{\varphi_i}=1,\, \forall\, i\in I_{\chi}\}$, we get
\begin{align*}
    P_{\kappa}(W_{f}=1|W_{\varphi_i} = \chi_i, \, \forall\, i=1,...,6) &= \frac{P_{\kappa}(W_{f}=1\cap\{W_{\varphi_i} = \chi_i, \, \forall\, i=1,...,6\})}{P_{\kappa}(W_{\varphi_i} = \chi_i, \, \forall i=1,...,6)}
    \\&\geq \frac{c_{3}P_{\kappa}\left(W_{f}=1\right)P_{\kappa}\left(W_{\varphi_i}=1,\, \forall\, i\in I_{\chi}\right)}{P_{\kappa}\left(W_{\varphi_i}=1,\, \forall\, i\in I_{\chi}\right)}
    \\&=c_{3}P_{\kappa}(W_{f}=1) = (1-\constNossoMu)c_{3}\P^{\Lambda}_{p,q}(\sigma_n(f)=1).
\end{align*}
Finally, summing over $\kappa$ yields the result.

\subsection{The multiscale} \label{SectionMultiscale}
Let $\mathbb{Q}^{\Lambda,n}_{p,q}$ be the percolation measure governing the sequence $\{Y_f\}_{f \in \E(\Z_+^2)}$, as defined in Proposition \ref{TeoDominacao}. The final step of the proof of Theorem \ref{mainTheorem} is to show that, for sufficiently large $n$, the renormalized model under $\mathbb{Q}^{\Lambda,n}_{p,q}$ satisfies \eqref{endMultiscale}. 

The percolation model on the square lattice examined in \cite{Marcos} is constructed as follows: columns are selected according to a renewal process in which the inter-arrival times have finite moment of order $1+\vartheta$, for some $\vartheta>0$. All vertical edges within selected columns and all horizontal edges are open with probability $p^*$. The remaining edges are closed. The authors of \cite{Marcos} show that percolation occurs whenever $p^*>p_0$, for some sufficiently large $p_0$. Their proof relies on a multiscale argument, constructing an infinite cluster through successive crossings of overlapping rectangles. The condition on the $1+\vartheta$ moment of the inter-arrival times ensures a decoupling inequality, allowing the authors to show that "good regions" are highly likely to occur. It is then shown that within these ``good regions" the probability of rectangle crossings becomes arbitrarily large, provided the parameter $p^*$ is sufficiently large.

Regarding our model, denote by $\Delta_i = \mathbbm{1}_{\{i \in \Lambda\}}$ the indicator function of the event that a column belongs to the environment $\Lambda$. For $\constDecaimento>2$, the renewal process generating $\Lambda$ has the property that $\mathbb{E}(\xi^{2})<\infty$. Hence, Lemma 2.1 of \cite{Marcos} shows that there exists a constant $c_{27} = c_{27}(\xi)$ such that, for all $m,r \in \Z_+$, and each pair of events $A,B$ measurable with respect to  $\sigma(\Delta_i:  0\leq i\leq m)$ and $\sigma(\Delta_i:  i\geq m+r)$, respectively, we have the decoupling inequality
\begin{equation}
\label{eq:Decoupling}
    \nu_\constDecaimento(A \cap B)\leq \nu_\constDecaimento(A)\nu_\constDecaimento(B)+c_{27}r^{-1}.
\end{equation}

Let $\Delta^n_i$ be the indicator that the interval $I_n^i$ is $\lambda$-good. Since the sequence $\{\Delta_i\}_{i\in \Z_+}$ is stationary, it follows that $\{\Delta^n_i\}_{i \in \Z_+}$ is also stationary. Also, note that if $A \in \sigma(\Delta^n_j: 0\leq j \leq m) = \sigma(\Delta_i: 0\leq i < (m+1)n)$ and $B \in \sigma(\Delta^n_j: j \geq m+r) = \sigma(\Delta_i:i \geq (m+r)n)$, then we may apply \eqref{eq:Decoupling} to show that  
\begin{equation}
\label{eq:DecouplingRenormalized}
        \nu_\constDecaimento(A \cap B)\leq \nu_\constDecaimento(A)\nu_\constDecaimento(B)+c_{27}(rn)^{-1}\leq \nu_\constDecaimento(A)\nu_\constDecaimento(B)+c_{27}r^{-1}.
\end{equation}


As in the multiscale scheme of \cite{Marcos}, let $\{L_m\}_{m \geq 1}$ be a sequence of numbers, henceforth called \first{scales}. Fix $\constMarcosAlfa \in (0, 1]$, $\constMarcosGamma \in \left(1,1+\tfrac{\constMarcosAlfa}{\constMarcosAlfa+2}\right)$, $\constMarcosMu \in \left(\tfrac{1}{\constMarcosGamma},1\right)$, and 
\begin{equation}\label{semideia}\constMarcosBeta \in (\constMarcosGamma \constMarcosMu-\constMarcosGamma +1,1).
\end{equation} Given an initial value $L_0$, define $\{L_m\}_{m\in\Z_+}$ recursively by
$$L_m = L_{m-1}\lfloor L_{m-1}^{\constMarcosGamma-1} \rfloor,\,\,m\geq 1.$$
For all $m\geq 0$, consider the intervals $J^m_i = [iL_m,(i+1)L_m)$ on the renormalized lattice.  Write $\ell_{i,m}= \{j \in \mathbb{N}: J_j^{m-1}\cap J_i^{m}\neq \emptyset\}$ and observe that $J_i^m = \bigcup_{j \in \ell_{i,m}} J_j^{m-1}$. 

We label each interval $J^m_i$ as \first{strong} or \first{weak}, inductively on $m\geq 0$. For $m=0$, the set $J^0_i$ consists of $L_0$ intervals of length $2n$ of the original lattice. We say the interval $J^0_i$ is strong if $I_n^j$ is $\lambda$-good for all $j\geq 0$ such that 
$I_n^j \cap [iL_02n,(i+1)L_02n)\neq \emptyset$; otherwise, we call it weak. For $m\geq 1$, an interval $J^m_i$ will be labeled weak if at least two non-consecutive elements of $\{J^{m-1}_j: j \in \ell_{i,m}\}$ are weak. Otherwise, $J^m_i$ will be called strong. 

Our definition of a strong interval in the first scale requires that all columns are $\lambda$-good in that interval (which is possible by taking $n$ large), while in \cite{Marcos} the equivalent notion requires only that at least one column is selected. Such a stronger requirement is necessary to show, with the aid of Propositions \ref{theoremEdge} and \ref{TeoDominacao}, that all edges in the block $J_i^0$ are open with high probability. However, our notion of strong intervals does not imply that our model dominates stochastically the one in \cite{Marcos}. In weak blocks, the probability of opening renormalized edges in our model, although positive, may be small, whereas in \cite{Marcos} all horizontal edges are open with probability $p^*$, which will be close to $1$. Nevertheless, we shall see that this can be circumvented with some effort.

 Denote by $p_{G}$ and $p_{B}$ the probabilities that favored and unfavored edges are open, respectively, in the independent model given by $\{Y_f\}_{f \in \E(\Z^2_+)}$; see \eqref{indep_measure}. According to Proposition \ref{theoremEdge}, an appropriate choice of $\lambda$ gives $p_G \to 1$ as $n\to \infty$ (see discussion after Lemma \ref{cotaMarcos}). Note also that although $p_B \to 0$ as $n \to \infty$, we still have $p_B\geq (c_{28})^n$ for some constant $c_{28}>0$ (see \eqref{BlockConstruction} and Proposition \ref{TeoDominacao}).

Let $L_0$ be large enough to satisfy the following properties:
\begin{itemize}
        \item[i)] $L_0^{\constMarcosGamma-1}\geq 3$,
    \item[ii)] $L_0^{c_{29}}>1+c_{27}$, where $c_{29} = 2+2\constMarcosAlfa-\constMarcosGamma\constMarcosAlfa-2\constMarcosGamma$.
\end{itemize}
With this choice of initial scale $L_0$, the probability of finding weak intervals decays rapidly. Fix $\constDecaimento_1>2$ such that $\frac{1+\constMarcosAlfa}{\frac{1}{2}\constDecaimento_1-1}<\constMarcosGamma \constMarcosMu-1$ and write 
$$p_m=\nu_{\constDecaimento_1}(J_i^m\mbox{ is weak}),\,\,m\geq 0.$$
We have the following lemma:
\begin{lema}\label{LemaDoMarcos}Let $\constDecaimento>\constDecaimento_1$ and $\lambda>\frac{1}{\constDecaimento}$. There exists $n$ large enough such that
$$p_m\leq L_m^{-\constMarcosAlfa},$$
for every $m\in\Z_+$.
\end{lema}

\begin{proof} The proof follows by induction. For $m=0$, the union bound and Proposition \ref{goodBlocksProposition} gives
$$\nu_{\constDecaimento}(J_i^0 \mbox{ is weak})=\nu_{\constDecaimento}(J_0^0 \mbox{ is weak}) =\nu_{\constDecaimento}\left(\bigcup_{j=0}^{L_0-1} \{I_n^j \mbox{ is }\lambda\mbox{-bad}\}\right)\leq 3cL_0n^{1-\lambda\constDecaimento}.$$
Since $\lambda>\tfrac{1}{\constDecaimento}$, the claim follows taking \begin{equation}n>\frac{1}{3c}L_0^{\frac{1+\constMarcosAlfa}{\lambda \constDecaimento-1}}.\label{condicaoAmbiente}\end{equation}
Under the conditions (i) and (ii) on $L_0$, the induction step follows exactly as in \cite{Marcos}, replacing their Lemma 2.1 by the decoupling in \eqref{eq:DecouplingRenormalized}.
\end{proof}

Next, we show that the probability of certain crossing events in strong regions is large. Before presenting the details, we need to introduce further notation. Let $H_0 = 2\exp(L_0^\constMarcosMu)$ and $H_m = 2 \lceil \exp(L_m^\constMarcosMu) \rceil H_{m-1}$,  $m\geq 1$. Consider the crossing events (see also Equations (33) and (34) in \cite{Marcos})
\begin{equation}
    C_{i,j}^m = H((J_i^m\cup J_{i+1}^m)\times [jH_m, (j+1)H_m)),
\end{equation}
and
\begin{equation}
D_{i,j}^m = V(J_i^m \times [jH_m, (j+2)H_m)).
\end{equation}
To simplify notation, denote by $\P_{p_G,p_B}$ the percolation measure under which favored (unfavored) edges are open with probability $p_G$ ($p_B$), independently. Define the function $q_m(p_G,p_B)$, as in Equation (36) of \cite{Marcos}, as
$$q_m(p_G,p_B) = \max\left\{\max_{\tiny\substack{\Lambda: J^m_j \mbox{ and }\\J_{i+1}^m \mbox{ are strong}}} \P_{p_G,p_B}((C_{i,j}^m)^c), \max_{\tiny\Lambda: J_i^m \mbox{ is strong}}\P_{p_G,p_B}((D^m_{i,j})^c)\right\}.$$

We need to prove that $q_m(p_G,p_B)\leq \exp(-L_m^\constMarcosBeta)$ for all $m\geq 0$, with $\beta$ as in \eqref{semideia}. We proceed by induction on $m$, with the initial step of the argument given by the following lemma.
\begin{lema}   \label{cotaMarcos} 
There exist $L^*\in\mathbb{N}$ and $s$ sufficiently close to 1, say $s>s_0$, such that
\begin{equation}q_0(s,p_B)\leq \exp(-L_0^\constMarcosBeta),\label{eq:LemaL0ePG}\end{equation}
for all $L_0>L^*$ and for all $p_B\in [0,1]$.
\end{lema}
\begin{remark}Observe that, for $m=0$, the quantity $q_0(p_G,p_B)$ is independent of $p_B$, as the events are related to favored edges only.
\end{remark}

\begin{proof}[Proof of Lemma \ref{cotaMarcos}]
Let $$\xi(s) =\lim_{t \to \infty }-\frac{t}{\log (\P_s(o \longleftrightarrow \partial B_t))}$$ be the correlation length at parameter $s$. A Peierls-type argument gives
\begin{align}\label{decreasing}
    q_0(s,p_B)&\leq 2\lceil \exp(L_0^\constMarcosMu) \rceil \P_{1-s}(\mbox{one-arm of length }L_0)\nonumber
    \\&\leq 2\lceil\exp(L_0^\constMarcosMu) \rceil e^{-\frac{1}{\xi(1-s)}L_0}.
\end{align}
For $L_0$ large enough, say $L_0>L^*$, the bound in \eqref{decreasing} is decreasing in $L_0$. Moreover, taking $s$ large enough (depending on $L^*$) we find
$$2\lceil\exp(L_0^\constMarcosMu) \rceil e^{-\frac{1}{\xi(1-s)}L_0}\leq \exp(-L_0^\constMarcosBeta),$$
as required.
\end{proof}

We shall make $p_G$ large enough to satisfy Lemma \ref{cotaMarcos} by choosing $n$ appropriately. Note that, unlike \cite{Marcos}, when $n$ increases, the parameter $p_B$ decreases. Thus, fine control of the size $n$ is needed.



The relationship between the parameter $p_B$ and the crossing events of \cite{Marcos} is described in Equation (43) of \cite{Marcos}, where it is required that 

\begin{equation*}
    \exp\left[L_k^{\constMarcosGamma \constMarcosBeta}-\exp\left(9\ln(p_B)L_k+\left(\frac{2}{3}\right)^\constMarcosMu L_k^{\constMarcosGamma \constMarcosMu}\right)\right]<1.
\end{equation*}
Our bound $p_B\geq (c_{28})^n$ yields the constraint 
\begin{equation}
    n\leq c_{30}(L_0^{(1)})^{\constMarcosGamma \constMarcosMu-1},
    \label{conditionPercolacao}
\end{equation}
for some constant $c_{30} = c_{30}(\constMarcosGamma, \constMarcosAlfa, \constMarcosMu, \constMarcosBeta,c_{28})>0$.

If we are to succeed with the multiscale scheme of \cite{Marcos}, the values of $\lambda, \constDecaimento, L_0$, and $n$ must be selected such that both conditions \eqref{condicaoAmbiente} and \eqref{conditionPercolacao} are met. Once these conditions are satisfied, the argument follows unchanged, leading to the desired result. Let us explain how the parameters are chosen.

\textbf{Step 1:} Choose $L_0^{(1)}$ sufficiently large to satisfy conditions (i), (ii), and  
\begin{enumerate}
    \item[iii)] $\left(1-4(L_0^{(1)})^{\constMarcosGamma-1}e^{-(L_0^{(1)})^\constMarcosBeta}\right)>e^{-1}$,
    \item[iv)]$\frac{4}{1-e^{-1}}\exp\left((L_0^{(1)})^{\constMarcosMu \constMarcosGamma}+(L_0^{(1)})^{\constMarcosGamma}+(L_0^{(1)})^{\constMarcosBeta \constMarcosGamma}-\frac{(L_0^{(1)})^{\constMarcosBeta+\constMarcosGamma-1}}{24}\right)<1$.
\end{enumerate}
Requirements (iii) and (iv) are necessary for the induction step of the proof (see (42) and (45) in \cite{Marcos}).

\textbf{Step 2: }Consider the crossing events at the initial scale $L_0^{(1)}$ and let $s_0$ be large enough to satisfy Lemma \ref{cotaMarcos}. 
We need to show that, given $\varepsilon>0$, there exist $\lambda$, $\delta_n$, and $n$ sufficiently large such that 
$$p_G\coloneqq\mathbb{Q}^{\Lambda,n}_{p_c-\delta_n,p_c+\varepsilon}(Y_f=1)>s_0.$$

Let $\constNossoBeta_0$ be such that $g(\constNossoBeta_0)>s_0$ and $\constNossoMu$ sufficiently small such that $\frac{\constNossoBeta_0}{1-\constNossoMu}<1$. By Proposition \ref{TeoDominacao}, it suffices to show that 
$$\mathbb{Q}^{\Lambda,n}_{p_c-\delta_n,p_c+\varepsilon}(\sigma_n(f)=1)>\frac{\constNossoBeta_0}{1-\constNossoMu},$$ 
for $n$ sufficiently large. By Proposition \ref{theoremEdge}, for each $\varepsilon>0$, there exists $\lambda=\lambda(\constNossoBeta_0)>0$, $\delta_n(\varepsilon,\lambda)>0$, and a large $n_1$ such that 
\begin{equation}\mathbb{Q}^{\Lambda,n}_{p_c-\delta_n,p_c+\varepsilon}(W_f=1)=(1-\constNossoMu)\mathbb{Q}^{\Lambda,n}_{p_c-\delta_n,p_c+\varepsilon}(\sigma_n(f)=1)>\constNossoBeta_0,\label{DefBeta0}\end{equation}
for all $n\geq n_1$.

\textbf{Step 3: } Given $\constNossoBeta_0, \constNossoMu, \lambda, \delta_n$, and $n_1$ fixed in the previous step, choose $\constDecaimento>\constDecaimento_1$ such that $\lambda>\frac{1}{\constDecaimento}$ and $\frac{1}{3c}(L_0^{(1)})^{\frac{1+\constMarcosAlfa}{\lambda \constDecaimento -1}}<c_{30}(L_0^{(1)})^{\constMarcosGamma \constMarcosMu-1}$. Let $I =  \left(\frac{1}{3c}(L_0^{(1)})^{\frac{1+\constMarcosAlfa}{\lambda \constDecaimento-1}}, c_{30}(L_0^{(1)})^{\constMarcosGamma \constMarcosMu-1}\right)$ and choose $L_0^*$ and $n$ according to the following algorithm:
\begin{enumerate}
    \item  If $n_1 \in I$, take $L_0^*=L_0^{(1)}$ and $n=n_1$,
    \item If $n_1<\frac{1}{3c}(L_0^{(1)})^{\frac{1+\constMarcosAlfa}{\lambda \constDecaimento-1}}$, take $L_0^* = L_0^{(1)}$ and choose $n>n_1$ such that $n\in I$.
    \item If $n_1>c_{30}(L_0^{(1)})^{\constMarcosGamma \constMarcosMu-1}$, take $n=n_1$ and choose $L_0^*>L_0^{(1)}$ such that $$n_1 \in \left(\frac{1}{3c}(L_0^*)^{\frac{1+\constMarcosAlfa}{\lambda \constDecaimento-1}}, c_{30}(L_0^*)^{\constMarcosGamma \constMarcosMu-1}\right).$$
\end{enumerate}

In Step 2, we selected parameters to obtain $p_G>s_0$ for crossing events at the initial scale $L_0^{(1)}$. Note that we may take $L_0^*>L_0^{(1)}$ and still have \eqref{eq:LemaL0ePG}. Hence, Step 3 does not affect the argument.


With these parameter choices, conditions \eqref{condicaoAmbiente} and \eqref{conditionPercolacao} are met. As a result, the induction argument in \cite{Marcos} can be applied, showing that \eqref{endMultiscale} holds (as in Equation (48) of \cite{Marcos}), thereby completing the proof of Theorem \ref{mainTheorem}.

For the reader's convenience, we summarize the whole argument as follows (see Figure \ref{fluxogram}).  The multiscale scheme provides parameters $\constNossoBeta_0<1$ and $\constNossoMu>0$, ensuring that the independent model on $\Z^2_+$ will percolate, provided the environment meets certain conditions. These conditions are achieved through Proposition \ref{goodBlocksProposition}, with an appropriate selection of $\lambda$ and $\constDecaimento$, which are the final constants to be determined.

 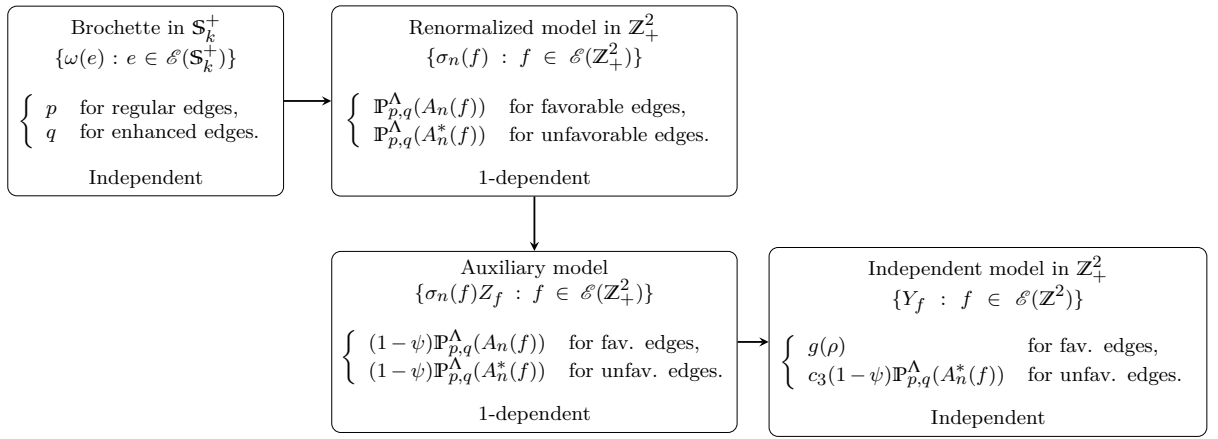
\begin{figure}[ht]
        \centering
        \resizebox{\linewidth}{!}{
            \begin{tikzpicture}[node distance=3cm]
                \node (brochette) [modSK,fill=white,text width=4cm] {Brochette in $\S_k^+$ \\$\{\omega(e): e \in \E(\S_k^+)\}$\vspace{0.3cm} \\ 
                $\left\{\begin{array}{ll}
                     p&  \mbox{for regular edges},\\
                     q& \mbox{for enhanced edges}. 
                \end{array}\right.$\vspace{0.3cm}\\Independent};               
                \node (renormalizado) [modZ2Dep, fill=white,right of=brochette,xshift=3cm, text width=6cm] {Renormalized model in $\Z^2_+$ \\ $\{\sigma_n(f): f \in \E(\Z^2_+)\}$\vspace{0.3cm}\\
                $\left\{\begin{array}{ll}\P^\Lambda_{p,q}(A_n(f))&\mbox{for favorable edges,}\\ 
                \P^\Lambda_{p,q}(A^*_n(f))&\mbox{for unfavorable edges.}\end{array}\right.$\vspace{0.3cm}\\1-dependent};
                \node (Wz) [modZ2Dep,fill=white,below of=renormalizado,text width=6cm,yshift=-0.75cm] {Auxiliary model \\ $\{\sigma_n(f)Z_f: f \in \E(\Z^2_+)\}$\vspace{0.3cm}\\$\left\{\begin{array}{ll}(1-\constNossoMu)\P^\Lambda_{p,q}(A_n(f))& \mbox{for fav. edges},\\(1-\constNossoMu)\P^\Lambda_{p,q}(A^*_n(f))&\mbox{for unfav. edges.}\end{array}\right.$\vspace{0.3cm}\\1-dependent};
                \node (Y) [modZ2inDep,fill=white,right of=Wz,xshift=4cm,text width=6.5cm] {Independent model in $\Z^2_+$ \\$\{Y_f: f \in \E(\Z^2)\}$\vspace{0.3cm}\\$\left\{\begin{array}{ll}g(\constNossoBeta)& \mbox{for fav. edges,}\\c_3(1-\constNossoMu)\P^\Lambda_{p,q}(A^*_n(f))& \mbox{for unfav. edges.}\end{array}\right.$\vspace{0.3cm}\\Independent};

                \draw [arrow] (brochette)--(renormalizado);
                \draw [arrow] (renormalizado)--(Wz);
                \draw [arrow] (Wz) -- (Y);
            \end{tikzpicture}
        }
        \caption{Chain of models showing edge-opening probabilities and edge dependencies for each case.}
        \label{fluxogram}
    \end{figure}

Using Proposition \ref{theoremEdge}, we identify parameters $\lambda_0>0$, $\delta>0$, and $n\geq 1$, ensuring that the event $A_n$ occurs with high probability to satisfy the requirements for stochastic domination. Once these values are established, $\constDecaimento$ is chosen large enough so that $\frac{1}{\constDecaimento}<\lambda_0$. This allows us to select $\lambda \in \left(\frac{1}{\constDecaimento},\lambda_0\right)$, satisfying both Propositions \ref{goodBlocksProposition} and \ref{theoremEdge}, making the environment suitable for the multiscale argument.

With these parameter choices, the multiscale argument shows that the independent model percolates, which, by stochastic domination, implies that the renormalized model also percolates. Finally,  the block construction shows that the original model percolates, concluding the argument.

\section*{Acknowledgements}  
Matheus B. Castro was partially supported by Fundação Coordenação de Aperfeiçoamento de Pessoal de Nível Superior (CAPES) and PPGMAT-UFMG. Rémy Sanchis was partially supported by Conselho Nacional de Desenvolvimento Científico e Tecnológico (CNPq), and Fundação de Amparo à Pesquisa do Estado de Minas Gerais (FAPEMIG), grants APQ-00868-21 and RED-00133-21. Roger Silva was partially supported by FAPEMIG, grant APQ-06547-24.

\section*{Author contributions}
All authors wrote and reviewed the manuscript.

\section*{Data Availability}
Data sharing is not applicable to this article as no datasets were generated or analysed during the current study.

\section*{Conflict on interest}
The authors declare no Conflict of interest.


\begin{thebibliography}{99}

        \bibitem{AizenmannGrimmett} {\sc Aizenman M. and Grimmett G.}: Strict monotonicity for critical points in percolation and ferromagnetic models. \textit{J. Stat. Phy.},  \textbf{63}, 817-835, 1991.

         \bibitem{A} {\sc Andjel E.D.}: Survival of multidimensional contact process in random environments. \textit{Bol. Soc. Brasil. Mat.},  \textbf{23}, 109-119, 1992.

          \bibitem{BS} {\sc Basu D. and Sapozhnikov A.}: Crossing probabilities for critical Bernoulli percolation on slabs. \textit{Ann. Inst. Henri Poincare},  \textbf{53}, 1921-1933, 2017.

          \bibitem{BK} {\sc Berg J. van den and Kesten H}:  On a combinatorial conjecture concerning disjoint occurrences of events. \textit{Ann. Probab.},  \textbf{15}, 354-374, 1987.

  \bibitem{bollobas} {\sc Bollobás, B. and Riordan, O.}: Percolation, Cambridge University Press, New York, 2006.

         \bibitem{Chayes}{\sc Borgs C., Chayes J.T. and Randall D.}:  The van den Berg-Kesten-Reimer inequality: A review. In Perplexing Problems in Probability. Festschrifft in honor of Harry Kesten. Birkhäuser, Boston, MA, 1999.
        
         \bibitem{CC} {\sc Chayes, J. T., Puha, A. and Sweet, T.}:  Independent and dependent percolation, Probability Theory and Applications, IAS/Park City mathematical series 6, 51-118, 1999..

         \bibitem{BDS} {\sc Bramson M., Durrett R. and Schonmann R.}: The contact process in random environment. \textit{Ann. Probab.},  \textbf{19}, 960-983, 1991.
        
        \bibitem{BH}{\sc Broadbent S.R. and Hammersley J.M.}: Percolation processes: I. crystals and mazes. \textit{Math. Proc. Cambridge}, \textbf{53}, 629-641, 1957.

         \bibitem{CK}{\sc Campanino M. and Klein A.}: Decay of two-point functions for $(d+1)$-dimensional percolation, Ising and Potts models with d-dimensional disorder. \textit{Commun. Math. Phys.}, \textbf{135}, 483-497, 1991. 

         \bibitem{CKP}{\sc Campanino M., Klein A. and Perez J.F.}: Localization in the ground state of the Ising model with a random transverse field. \textit{Commun. Math. Phys.}, \textbf{135}, 499-515, 1991. 

         \bibitem{DST}{\sc Duminil-Copin H., Sidoravicius V. and Tassion V.}: Absence of infinite cluster for critical Bernoulli percolation on slabs. \textit{Comm. Pure Appl. Math.}, \textbf{69}, 1397-1411, 2016.
        
        \bibitem{Brochette}{\sc Duminil-Copin H., Hilário M.R., Kozma G. and Sidoravicius V.}: Brochette Percolation. \textit{Isr. J. Math.},\textbf{225}, 479-501, 2018.

        \bibitem{DKT}{\sc Duminil-Copin H., Kozma G. and Tassion V.}: Upper bounds on the percolation correlation length . \textit{In M. Eulália Vares, R. Fernández, L. Renato Fontes, \& C. M. Newman (Eds.), Progress in Probability}, Birkhauser, 347-369, 2021.

        \bibitem{F}{\sc Fisher M.E.}: Scaling, universality and renormalization group theory. In Critical Phenomena \textit{Lecture Notes in Phys.}, \textbf{186}, 1-139, 1983.

        
        
       

       
        
        \bibitem{GrimmettMarstrand}{\sc Grimmett G.R. and  Marstrand J.M.}: The Supercritical Phase of Percolation is Well Behaved. \textit{Proc. Roy. Soc. London Ser. A}, \textbf{430}, 439–457, 1990.


        \bibitem{Marcos}  {\sc Hilário M.,  Sá M., Sanchis R. and Teixeira A.}: Phase transition for percolation on a randomly stretched square lattice. \textit{Ann. Appl. Probab.}, \textbf{33}, 3145-3168, 2023.

        \bibitem{Marcos2}  {\sc Hilário M.,  Sá M., Sanchis R. and Teixeira A.}: A new proof for percolation phase transition on stretched lattices. \textit{arXiv:2311.14644}, 2024.
        
        \bibitem{H}{\sc Hoffman C.}: Phase transition in dependent percolation. \textit{Commun. Math. Phys.}, \textbf{254}, 1-22, 2005. 

         \bibitem{JMP}{\sc Jonasson J., Mossel E. and Peres Y.}: Percolation in a dependent random environment. \textit{Random Struct. Algor.}, \textbf{16}, 333-343, 2000. 

        
        \bibitem{Kesten1987}{\sc Kesten H.}: Scaling relations for 2d-percolation. \textit{Commun. Math. Phys.}, \textbf{109}, 109-156, 1987.  
        
        \bibitem{KSV}{\sc Kesten H., Sidoravicius V. and Vares M.E.}: Oriented percolation in a random environment. \textit{Electron. J. Probab.}, \textbf{27}, 1-49, 2022.

         \bibitem{K}{\sc Klein A.}: Extinction of contact and percolation processes in random environment. \textit{Ann. Probab.}, \textbf{16}, 333-343, 2000. 

         \bibitem{L}{\sc Liggett T.M.}: The survival of one-dimensional contact processes in random environments. \textit{Ann. Probab.}, \textbf{20}, 696-723, 1992.

        \bibitem{LSS}{\sc Liggett T.M., Schonmann R. H. and Stacey A. M.}: Domination by product measures. \textit{Ann. Probab.}, \textbf{25}, 71-95, 1997.

        \bibitem{MW}{\sc McCoy B.M., Wu T.T.}: Theory of a two-dimensional Ising model with random impurities. I. Thermodynamics. \textit{Phys. Rev.}, \textbf{176}, 631-643, 1968. 

        \bibitem{tassion}{\sc Newman C.M., Tassion V. and Wu W.}: Critical percolation and the minimal spanning tree in slabs. \textit{Comm. Pure Appl. Math.} \textbf{70}: 2084-2120,  2017. 

        \bibitem{NV}{\sc Newman C.M., Volchan S.B.}: Persistent survival of one-dimensional contact processes in random environments. \textit{Ann. Probab.} \textbf{24}: 411-421,  1996. 
  
        \bibitem{Nolin}{\sc Nolin P.}: Near-critical percolation in two dimensions. \textit{Electron. J. Probab.}, \textbf{13}, 1562-1623, 2008. 

        \bibitem{R}{\sc Reimer D.}:  Proof of the Van den Berg-Kesten conjecture. \textit{Combin Probab Comput}, \textbf{9}, 27-32, 2000.
        
        \bibitem{Russo1978}{\sc Russo L.}: A note on percolation. \textit{Zeitschrift f{\"u}r Wahrscheinlichkeitstheorie und Verwandte Gebiete}, \textbf{43}, 39-48, 1978.

           \bibitem{S}{\sc Stauffer D.}: Scaling theory of percolation clusters. \textit{Phys. Rep.}, \textbf{54}, 1-74, 1979.
        
        \bibitem{SEYMOUR1978227}{\sc Seymour P. and Welsh D.}: Percolation probabilities on the square lattice. \textit{Ann. Discrete Math.}, \textbf{3}, 227-245, 1978.

        \bibitem{Z}{\sc Zhang Y.}: A note on inhomogeneous percolation. \textit{Ann. Probab.}, \textbf{22}, 803-819, 1994.        

     
\end{thebibliography}
\end{document}